\documentclass[oneside,english, 11pt]{amsart}
\usepackage[T1]{fontenc}
\usepackage[utf8]{inputenc}

\usepackage{amsthm}
\usepackage{amstext}
\usepackage{amssymb}
\usepackage{amsmath}
\usepackage{mathtools}
\makeatletter
\numberwithin{equation}{section}
\numberwithin{figure}{section}
\theoremstyle{plain}
\newtheorem{thm}{\protect\theoremname}
  \theoremstyle{plain}
  \newtheorem{prop}[thm]{\protect\propositionname}
  \theoremstyle{remark}
  \newtheorem{rem}[thm]{\protect\remarkname}
  \theoremstyle{definition}
  \newtheorem{defn}[thm]{\protect\definitionname}
  \theoremstyle{plain}
  \newtheorem{lem}[thm]{\protect\lemmaname}
    \theoremstyle{definition}
  \newtheorem{example}[thm]{\protect\examplename}

\theoremstyle{remark}

\usepackage[english,french]{babel}
  \providecommand{\definitionname}{Definition}
  \providecommand{\lemmaname}{Lemma}
  \providecommand{\propositionname}{Proposition}
  \providecommand{\remarkname}{Remark}
\providecommand{\theoremname}{Theorem}
\providecommand{\examplename}{Example}
\providecommand{\assertionname}{Assertion}



\author{Rudy Rosas}
\thanks{The author was supported by the Vicerrectorado the Investigaci\'on de la Pontificia Universidad Cat\'olica del Per\'u}

\email{rudy.rosas@pucp.pe}

\address{Pontificia Universidad Cat\'olica del Per\'u, Av Universitaria 1801, Lima, Per\'u.}

\begin{document}

\title{Nodal separators of holomorphic
foliations}

\maketitle

\selectlanguage{english}

\begin{abstract}
We study a special kind of local invariant sets of singular holomorphic foliations called nodal separators \cite{MM,camacho2013}. We define  notions of equisingularity and topological equivalence for nodal separators as intrinsic objects and,  in analogy with the celebrated theorem of Zariski for analytic curves, we prove the equivalence of these notions. We give some applications in the study of topological equivalences of holomorphic foliations. In particular, we show that the nodal singularities and its eigenvalues in the resolution of a generalized curve  are topological invariants.
\end{abstract}

\tableofcontents
\section{Introduction}
We consider a one-dimensional holomorphic foliation $\mathcal{F}$ on a complex smooth surface $V$, with an isolated singularity
at $p\in V$. In local coordinates $(\mathbb{C}^2,0)\simeq (V,p)$ the foliation is generated by a holomorphic vector field $Z$ with an isolated singularity at $0\in\mathbb{C}^2$.  The singularity at $p\in V$ is called {\it reduced} if the linear part of $Z$ has eigenvalues
$\lambda_{1},\lambda_{2}\in\mathbb{C}$ with $\lambda_{1}\neq0$ and such that
$\lambda=\frac{\lambda_{2}}{\lambda_{1}}$ is not a rational positive number.
This last number will be called the eigenvalue of the singularity $p\in V$. The singularity $p$ is {\it hyperbolic}  if
$\lambda\in\mathbb{C}\backslash\mathbb{R}$,
it is a {\it saddle} if $\lambda<0$, it is a {\it node} if $\lambda\in(0,\infty)\backslash \mathbb{Q}$, and it
is a {\it saddle-node} if $\lambda=0$. When the singularity of $\mathcal{F}$
at $p\in V$ is a node we have a particular kind of local invariant
sets:  In this case there are suitable local coordinates such that the foliation
near $p\in V$ is given by the holomorphic vector field
$x\frac{\partial}{\partial x}+\lambda y\frac{\partial}{\partial y}$
and we have the multi-valued first integral $yx^{-\lambda}$. Then
the closure of any leaf other than the separatrices is a set of
type $|y|=c|x|^{\lambda}$ ($c>0$) which is called a {nodal separator}
\cite{MM}. More precisely, we say that a set $S$ is a {\it nodal separator} for a node, if in
linearizing coordinates as above we have $S=\{(x,y):|y|=c|x|^{\lambda}\}\cap B$, $c>0$, where  $B$ is
an open ball centered at the singularity. Clearly $S$ is  invariant by the foliation restricted to $B$. In general, if the singularity at $p\in V$ is not necessarily reduced, we say that a set $S\subset V$ is a {\it nodal separator at p}
if there is a neighborhood $U$ of $p$ in $V$ such that the strict transform of $S\cap U$ in the resolution of $\mathcal{F}$ is a nodal separator for some node in the resolution.
The nodal separators and the separatrices are the minimal dynamical blocks at a singularity, as the  following theorem asserts \cite{camacho2013}.

\begin{thm}\label{cr}
Let $\mathcal{F}$ be a germ of holomorphic foliation with  an isolated singularity at $0\in\mathbb{C}^2$.
 Let $\mathcal{I} $ be a closed connected invariant set such that $\{0\}\subsetneq \mathcal{I}$. Then $\mathcal{I}$ contains either a separatrix
or a nodal separator at $0\in\mathbb{C}^2$. In particular,  if $L$ is  a local leaf of $\mathcal{F}$ such that $0 \in \overline{L}$, then $\overline{L}$ contains either a   separatrix or a nodal separator  at $0 \in \mathbb{C}^{2}$.
\end{thm}
In this paper, we study some properties of  nodal separators at $(\mathbb{C}^2,0)$ as intrinsic objects, that is,  not necessarily linked to a holomorphic foliation at  $(\mathbb{C}^2,0)$. The nodal separators have a good behavior under complex blow ups:  these object has well defined iterated tangents and so, in an infinitesimal viewpoint, they look like curves, although the information given by the sequence of infinitely near points in the case of nodal separators is essentially infinite. However, in analogy with the case of curves, in Section \ref{nodal separators} we establish the concept of equisingularity for nodal separators. On the other hand, also in Section \ref{nodal separators} we give a notion of topological equivalence for nodal separators: roughly speaking, we say that two nodal separators $S$ and $S'$ at $(\mathbb{C}^2,0)$ are topologically equivalent if there is a local homeomorphism of  the ambient space  taking  $S$ to $S'$ and preserving the ``Levi foliations'' defined on $S$ and $S'$.  
The following theorem, which is one of the main results of this work, is analogous to a well known theorem for curves due to Zariski \cite{zariski}.
\begin{thm}\label{Zariski}  Two nodal separators are equisingular if and only if they are topologically equivalent.
\end{thm}
The proof of this theorem is organized as follows. In Section \ref{eq-top} we prove the first part of Theorem \ref{Zariski}: equisingularity implies topological equivalence. In Section \ref{top-eq} we reduce the second part of Theorem \ref{Zariski} to  Proposition \ref{aproximacion por curvas}.  We begin the proof of Proposition \ref{aproximacion por curvas} in section \ref{better} with the construction of a ``nice'' topological equivalence (Proposition \ref{nice}). Finally, we end the proof of Proposition \ref{aproximacion por curvas} in Section \ref{proofAC}.

In the context of holomorphic foliations   at $(\mathbb{C}^2,0)$, in Section \ref{fol1} we prove the following theorem.
\begin{thm}\label{bijection} Let  $\mathcal{F}$ and $\widetilde{\mathcal{F}}$ be holomorphic foliations with isolated singularities at $0\in\mathbb{C}^2$. Let $\mathfrak{h}:{\mathfrak{U}}\rightarrow\widetilde{{\mathfrak{U}}}$, $\mathfrak{h}(0)=0$  be a topological equivalence between  $\mathcal{F}$ and $\widetilde{\mathcal{F}}$. Then there is a bijection $\mathfrak{h}_*$ between the set $\mathfrak{N}$ of nodes in the resolution of $\mathcal{F}$ with the set  $\widetilde{\mathfrak{N}} $ of nodes in the resolution of $\widetilde{\mathcal{F}}$ such that: the nodal separators issuing from a node $\mathfrak{n}\in \mathfrak{N}$ are mapped to the nodal separators issuing from the node $\mathfrak{h}_*(\mathfrak{n})\in\widetilde{\mathfrak{N}}$. In particular, the number of nodes in the resolution of a foliation is a topological invariant.
\end{thm}
Observe that this theorem does not need any hypothesis on the foliations. In particular, the foliations could have saddle-nodes in its resolutions, so Theorem  \ref{bijection} is really new outside the class of generalized curves \cite{CSL}. In the case of \emph{Generic General Type} foliations,  Theorem \ref{bijection} is a consequence of the work of Mar\'in and Mattei \cite{MM} --- Generic General Type foliations are generalized curves  with an additional generic dynamical property which guarantees that the conjugation $\mathfrak{h}$ is transversely holomorphic ---. In fact, in \cite{MM} the authors prove much more: if $\mathcal{F}$ is of Generic General Type and $\widetilde{\mathcal{F}}$ is any foliation topologically equivalent to $\mathcal{F}$, then there exists a topological equivalence between  $\mathcal{F}$ and $\widetilde{\mathcal{F}}$  extending to the exceptional divisor after the resolutions of $\mathcal{F}$ and $\widetilde{\mathcal{F}}$. On the other hand, if $\mathcal{F}$ is a generalized curve  not necessarily of Generic General Type, in \cite{rosas3} is proved that always exists a topological equivalence between $\mathcal{F}$ and $\widetilde{\mathcal{F}}$  extending after resolution to a neighborhood of each linearizable or resonant singularity which is not a corner. In particular,
this topological equivalence extends to each nodal singularity which is not a corner. The goal of the last theorem of this paper, proved in Section \ref{fol1},  is to construct a topological equivalence extending also to the nodal singularities in the corners of the resolution:  

\begin{thm}\label{nodalext} Let $\mathcal{F}$ and $\widetilde{\mathcal{F}}$ be topological equivalent holomorphic foliations at  $(\mathbb{C}^2,0)$. Suppose that $\mathcal{F}$ is a generalized curve. Then there exists a topological equivalence between $\mathcal{F}$ and $\widetilde{\mathcal{F}}$ which, after resolution, extends as a homeomorphism to a neighborhood of each linearizable or resonant non-corner singularity and each nodal corner singularity. In particular, the eigenvalue of each nodal singularity in the resolution of $\mathcal{F}$ is a topological invariant.
\end{thm} 
 A key step in the proof of this theorem  is to establish a correspondence, after resolution,  between the singularities of $\mathcal{F}$ and  $\widetilde{\mathcal{F}}$. When a singularity $p$ in the resolution of $\mathcal{F}$ is not a corner, we can use the separatrix issuing from $p$ to  define the corresponding singularity $\tilde{p}$ in the resolution of $\widetilde{\mathcal{F}}$. Moreover, By Zariski's Theorem \cite{zariski},  the singularities $p$ and $\tilde{p}$ are in ``isomorphic positions'' in their corresponding exceptional divisors. If the singularity $p$ is a corner, we have no separatrix issuing from $p$ and this is the main difficulty when we deal with corner singularities  --- recall that  $\mathcal{F}$ is not necessarily of Generic General Type, so the techniques of \cite{MM} does not work ---. However, if the corner singularity $p$ is a node, we can overcome this difficulty by using a nodal separator issuing from $p$ and Theorem \ref{bijection} to define the singularity $\tilde{p}$  corresponding to $p$ in the resolution of $\widetilde{\mathcal{F}}$. Moreover,  Theorem \ref{Zariski}  guarantees that  $p$ and $\tilde{p}$ are in ``isomorphic positions'' in  their corresponding exceptional divisors. From this point  the construction of a topological equivalence extending to $p$ follows some ideas already used in \cite{rosas3}.

\section{Nodal separators}\label{nodal separators}

Let $V$ be a complex surface and let $p\in V$ be a  regular point. 

\begin{defn}  A set $S\subset V$ will be called a nodal separator at $p\in V$ if there exist
\begin{enumerate}
\item a complex surface $M$;
\item a map $\pi:M\rightarrow V$, which is a finite composition of blow ups at points equal or infinitely near to $p\in V$; and
\item a germ of nodal foliation $\mathcal{F}$ at some point $q\in\pi^{-1}(p)$ 
\end{enumerate} such that the strict transform of $S$ by $\pi$ is a nodal separator of $\mathcal{F}$ at  $q\in M$. By simplicity, we will denote the strict transform of $S$ by $\pi$ also by $S$, so we can say that $S$  is a nodal separator of $\mathcal{F}$ at  $q\in M$.
\end{defn}

\begin{rem}\label{corner}
In the definition above, by performing additional blow ups at $q$ if necessary, we can assume the following additional properties:
\begin{enumerate}\item the point $q$ is the intersection of two irreducible components $E_1$ and $E_2$ of the exceptional divisor $\pi^{-1}(p)$;
\item $E_1$ and $E_2$ are the separatrices of the nodal foliation $\mathcal{F}$ at $q\in M$.
\end{enumerate}
\end{rem}
\begin{rem} Let $S$ be a nodal separator  at $p\in V$.  Restricted to some neighborhood of $p$, the nodal separator $S$ has the following properties:
\begin{enumerate}
\item $S$ is a real surface of dimension three with an isolated singularity at $p\in S$;
\item the Levi distribution on $S\backslash\{p\}$ is integrable, so we have a Levi foliation on  $S\backslash\{p\}$;
\item the Levi foliation on  $S\backslash\{p\}$ is minimal, that is,  its leaves are dense in $S$.
\end{enumerate}
At this point, the following question  become interesting: there exist other examples of real surfaces  satisfying  the properties 1,2 and 3 above? We can easily construct examples which are, essentially, immersed nodal separators: Let $S$ be a nodal separator at $p\in V$ and let $\psi:S\rightarrow V$, $\psi(p)=p$ be continuous, injective and holomorphic on a neighborhood of $S\backslash\{p\}$; then $\psi(S)$ satisfies properties 1,2 and 3 above. There exists an essentially different example?
\end{rem}
  
As in the case of germs of curves, we will define a notion of equisingularity for nodal separators. Let $S$ be a nodal separator  at $p\in V$. We denote by $\mathcal{N}_p(S)$ the set of points equal or infinitely near to $p$ that lie on $S$. 

\begin{defn} Let $V$ and $\widetilde{V}$ be smooth surfaces and let $S$ and $\widetilde{S}$ be two nodal separators at $p\in V $ and at $\tilde{p}\in \widetilde{V}$, respectively. We say that $S$ and $\widetilde{S}$ are equisingular if  there exists a bijection $\phi:\mathcal{N}_p(S)\rightarrow \mathcal{N}_{\tilde{p}}(\widetilde{S})$  preserving the natural ordering and proximity of infinitely near points, that is: $\zeta_1$ is infinitely near (resp. proximate) to $\zeta_2$ if and only if  $\phi(\zeta_1)$ is infinitely near (resp. proximate) to $\phi(\zeta_2)$.
\end{defn}

It is easy to see that, after a blow up at $p\in V$, the nodal separator $S$ intersects the exceptional divisor at exactly one point; clearly this property holds after successively blow ups. In other words,  there is a single point on $S$ in each infinitesimal neighborhood of $p$. Therefore the points in $\mathcal{N}_p(S)$  are sequentially ordered by the natural ordering of infinitely near points.

\begin{prop}\label{caso lineal}Let $S$ and   $\widetilde{S}$ be  nodal separators  associated to nodal singularities at $p\in V$ and $\tilde{p}\in\widetilde{V}$ of eigenvalues $\lambda$ and $\tilde{\lambda}$ in $(1,+\infty)\backslash \mathbb{Q}$, respectively.  Then, $S$ and $\widetilde{S}$ are equisingular if and only if $\lambda=\tilde{\lambda}$. 
\end{prop}
\begin{rem} Clearly, by taking the multiplicative inverse if necessary, we can assume that the eigenvalue of a node belongs to $(1,+\infty)\backslash \mathbb{Q}$.
\end{rem}
\begin{proof} If $\lambda=\tilde{\lambda}$, in linearizing coordinates we have that $S$ and $\widetilde{S}$ are both nodal separators associated to the node    $x\frac{\partial}{\partial x}+\lambda y\frac{\partial}{\partial y}$.  This  implies the equisingularity of $S$ and $\widetilde{S}$.  Suppose now that $S$ and $\widetilde{S}$ are equisingular.  
Again, in linearizing coordinates $S$ is a nodal separator of the node $x\frac{\partial}{\partial x}+\lambda y\frac{\partial}{\partial y}$, so $S$ is given by $\{|y|=c|x|^\lambda\}$ for some $c>0$. Moreover, after the linear change of coordinates $(x,y)\mapsto (x,ry)$, for some $r>0$, we can assume that $c=1$.   Let $p_1, p_2,\ldots$ be  the points infinitely near to $p\in V$ that lie on $S$, that is:
\begin{enumerate}
\item $p_1$ is the only point in the exceptional divisor $E_1$  of the  blow up at $p\in V$, that lies in $S$;
\item $p_j$ is the only point in the exceptional divisor $E_j$ of the blow up at $p_{j-1}$, that lies in $S$ ($j\ge 2$).
\end{enumerate} 
All the strict transforms of $E_j$  by subsequent blow ups are also denoted by $E_j$. Define the sequence $n_1,n_2,\dots$ of natural numbers as follows:\begin{enumerate}
\item Let $n_1\in\mathbb{N}$ be  such that $p_1,\ldots,p_{n_1}\in E_1$ and $p_{n_1+1}\notin E_1$. It is not difficult to see that $n_1=[\frac{\lambda}{\lambda-1}]$, so $\frac{\lambda}{\lambda-1}=n_1+\frac{1}{\lambda_1}$ for some $\lambda_1>1$.
\item Let $n_2\in\mathbb{N}$ be such that $p_{n_1+1},\ldots,p_{n_1+n_2}\in E_{n_1+1}$ and $p_{n_1+n_2+1}\notin E_{n_1+1}$. In this case we have $n_2=[\lambda_1]$ and therefore $\frac{\lambda}{\lambda-1}=n_1+\frac{1}{n_2+\frac{1}{\lambda_2}}$ for some $\lambda_2>1$.
\item Let  $n_3\in\mathbb{N}$ be such that $p_{n_1+n_2+1},\ldots,p_{n_1+n_2+n_3}\in E_{n_1+n_2+1}$ and $p_{n_1+n_2+n_3+1}\notin E_{n_1+n_2+1}$. Then  $\frac{\lambda}{\lambda-1}=n_1+\frac{1}{n_2+\frac{1}{n_3+\frac{1}{\lambda_3}}}$ for some $\lambda_3>1$.
\item etc.
\end{enumerate}
Therefore $[n_1,n_2,\ldots]$ is the representation of $\frac{\lambda}{\lambda-1}$ as a continued fraction. On the other hand, let $\widetilde{p}_1, \widetilde{p}_2,\ldots$ be the points infinitely near to $\widetilde{p}$ that lies in $\widetilde{S}$:
\begin{enumerate}
\item $\widetilde{p}_1$ is the only point in the exceptional divisor $\widetilde{E}_1$  of the  blow up at $\widetilde{p}$, that lies in $\widetilde{S}$,
\item $\widetilde{p}_j$ is the only point in the exceptional divisor $\widetilde{E}_j$ of the blow up at $\widetilde{p}_{j-1}$, that lies in $\widetilde{S}$ ($j\ge 2$).
\end{enumerate}
Since the nodal separators ${S}$ and $\widetilde{S}$ are equisingular, clearly we have that
\begin{enumerate}
\item $\widetilde{p}_1,\ldots,\widetilde{p}_{n_1}\in \widetilde{E}_1$ and $\widetilde{p}_{n_1+1}\notin \widetilde{E}_1$.  So $\frac{\tilde{\lambda}}{\tilde{\lambda}-1}=n_1+\frac{1}{\tilde{\lambda}_1}$ for some $\tilde{\lambda}_1>1$.
\item $\widetilde{p}_{n_1+1},\ldots,\widetilde{p}_{n_1+n_2}\in \widetilde{E}_{n_1+1}$ and $\widetilde{p}_{n_1+n_2+1}\notin \widetilde{E}_{n_1+1}$. So $\frac{\tilde{\lambda}}{\tilde{\lambda}-1}=n_1+\frac{1}{n_2+\frac{1}{\tilde{\lambda}_2}}$ for some $\tilde{\lambda}_2>1$.

\item etc.
\end{enumerate}
From this we conclude that  $[n_1,n_2,\ldots]$ is also the representation of $\frac{\tilde{\lambda}}{\tilde{\lambda}-1}$ as a continued fraction, so $\widetilde{\lambda}=\lambda$.

\end{proof}

As in the case of curves, we can establish a notion of topological equivalence for nodal separators.
\begin{defn} Let $V$ and $\widetilde{V}$ be smooth surfaces and let $S$ and $\widetilde{S}$ be two nodal separators at $p\in V $ and at $\tilde{p}\in \widetilde{V}$, respectively. We say that $S$ and $\widetilde{S}$ are topological equivalent if there  is an orientation preserving homeomorphism $\mathfrak{h}:{\mathfrak{U}}\rightarrow\widetilde{{\mathfrak{U}}}$, $\mathfrak{h}(p)=\tilde{p}$  between neighborhoods of $p\in V$ and $\tilde{p}\in\widetilde{V}$, such that:
 \begin{enumerate}
 \item $\mathfrak{h}(S\cap \mathfrak{U})=\widetilde{S}\cap\widetilde{\mathfrak{U}}$;
     \item $\mathfrak{h}$ conjugates the Levi foliations of $S$ and $\widetilde{S}$.
          \end{enumerate}
          The homeomorphism $\mathfrak{h}$ will be called a topological equivalence between the nodal separators $S$ and $\widetilde{S}$.
\end{defn}
\begin{example}\label{ejemplonodo1} Two nodal separators of $x\frac{\partial}{\partial x}+\lambda y\frac{\partial}{\partial y}$, $\lambda\in(0,+\infty)\backslash \mathbb{Q}$ are topologically equivalent by a  biholomorphism of the form $(x,y)\mapsto(x,ry)$,  $r>0$. Thus, given a nodal separator $S$ of a nodal singularity, after a  holomorphic change of coordinates we can always assume that $S=\{|y|=|x|^{\lambda}\}$.

\end{example}

\section{Equisingularity implies topological equivalence}\label{eq-top}
In this section we prove the first part of Theorem \ref{Zariski}: equisingularity implies topological equivalence. Then, we assume that the nodal separators $S$ at $p\in V $ and $\widetilde{S}$ at $\widetilde{p}\in \widetilde{V} $  are equisingular.  Let $p_1, p_2,\ldots$ be the points infinitely near to $p$ that lie on $S$: \begin{enumerate}
\item $p_1$ is the only point in the exceptional divisor $E_1$  of the  blow up at $p$, that lies in $S$;
\item $p_j$ is the only point in the exceptional divisor $E_j$ of the blow up at $p_{j-1}$, that lies in $S$ ($j\ge 2$).
\end{enumerate} All the strict transforms of $E_j$  by subsequent blow ups are also denoted by $E_j$. Analogously, let $\tilde{p}_1\in\widetilde{E}_1$, $\tilde{p}_2\in\widetilde{E}_2,\ldots$ be the points infinitely near to $\tilde{p}$ that lie on $\widetilde{S}$. There exists $k\in\mathbb{N}$ such that  $S$ and $\widetilde{S}$ are nodal separators issuing from nodal foliations at $p_k$ and $\widetilde{p}_k$ respectively. By Remark \ref{corner}, if we take $k$ large enough we can assume the following properties:
\begin{enumerate}
\item $p_k$ is the intersection of $E_k$ with $E_l$ for some $l<k$;  
\item $\widetilde{p}_k$ is the intersection of $\widetilde{E}_k$ with $\widetilde{E}_{\widetilde{l}}$ for some $\widetilde{l}<k$;
\item $E_k$ and $E_l$ are the separatrices of the nodal foliation generating $S$;
\item $\widetilde{E}_k$ and $\widetilde{E}_{\widetilde{l}}$  are the separatrices of the nodal foliation generating $\widetilde{S}$.
\end{enumerate}
 By the equisingularity of  $S$ and $\widetilde{S}$ we have in fact that $\widetilde{l}=l$. 
From example \ref{ejemplonodo1}, we can take  local holomorphic coordinates $(x,y)$ at $p_k$ and  $(\widetilde{x},\widetilde{y})$ at $\widetilde{p}_k$  such that:

\begin{enumerate}
\item $E_l=\{y=0\}$, $E_k= \{x=0\}$;
\item $\widetilde{E}_l=\{\widetilde{y}=0\}$, $\widetilde{E}_k=\{\widetilde{x}=0\}$;
 \item $S=\{|y|=|x|^{\lambda}\}$;  
 \item $\widetilde{S}=\{|\widetilde{y}|=|\widetilde{x}|^{\widetilde{\lambda}}\}$.
\end{enumerate}
Observe that  $p_{k+1}\in E_l$ if and only if $\lambda>1$. On the other hand, by the equisingularity of  $S$ and $\widetilde{S}$ we have that $p_{k+1}\in E_l$ if and only if  $\widetilde{p}_{k+1}\in \widetilde{E}_l$. Then we deduce that $\lambda>1$ if and only if $\widetilde{\lambda}>1$. Without loss of generality we can assume that  $\lambda$ and $\widetilde{\lambda}$ are both greater than one.  Then, since the nodal separators $S$ at $p_k$ and $\widetilde{S}$ at $\widetilde{p}_k$ are also equisingular, from proposition \ref{caso lineal} we conclude that   $\lambda=\widetilde{\lambda}$. Let $M$ and $\widetilde{M}$ be the manifolds obtained by performing the $k$ successively blow ups at $p,p_1,...,p_{k-1}$ and at $\widetilde{p},\widetilde{p}_1,...,\widetilde{p}_{k-1}$, respectively.   Obviously, the homeomorphism $h$ from a neighborhood of $p_k$  to a neighborhood of $\widetilde{p}_k$  given by $h(x,y)=(x,y)$ is a topological equivalence between the nodal separators $S$ at $p_k$ and $\widetilde{S}$ at $\widetilde{p}_k$. This homeomorphism extends as a homeomorphism of a neighborhood of  $E_1\cup\ldots\cup E_k$ in $M$ to a neighborhood of   $\widetilde{E}_1\cup\ldots\cup \widetilde{E}_k$ in $\widetilde{M}$. Therefore the nodal separators $S$ at $p\in V $ and $\widetilde{S}$ at $\widetilde{p}\in \widetilde{V} $  are topologically equivalent.

\section{Topological equivalence implies equisingularity}\label{top-eq}

In this section we reduce the proof of Theorem \ref{Zariski} to the proof of Proposition \ref{aproximacion por curvas} stated below.  Naturally, we assume that the nodal separators $S$ and $\widetilde{S}$ are topologically equivalent.

 Let $p_1, p_2,\ldots$ the points infinitely near to $p$ that lie on $S$, that is:\begin{enumerate}
\item $p_1$ is the only point in the exceptional divisor $E_1$  of the  blow up at $p$, that lies in $S$;
\item $p_j$ is the only point in the exceptional divisor $E_j$ of the blow up at $p_{j-1}$, that lies in $S$ ($j\ge 2$).
\end{enumerate}
In the same way, we consider the sequence  $\tilde{p}_1\in\widetilde{E}_1$, $\tilde{p}_2\in\widetilde{E}_2,\ldots$ of points infinitely near to $\tilde{p}$ that lie on $\widetilde{S}$.  

\begin{prop}\label{aproximacion por curvas}Given $k\in\mathbb{N}$, there exist two germs of analytic irreducible curves $\mathfrak{ C}$ at $p$ and $\widetilde{\mathfrak{ C}}$ at $\widetilde{p}$  such that:

\begin{enumerate}
\item $\mathfrak{ C}$ and $\widetilde{\mathfrak{ C}}$ are topologically equivalent as inmersed curves;
\item the points $p_1,...,p_k$ lies in $\mathfrak{ C}$;
\item the points $\widetilde{p}_1,...,\widetilde{p}_k$ lies in $\widetilde{\mathfrak{ C}}$.
\end{enumerate}
\end{prop}
Since topological equivalence implies equisingularity in the case of curves,   it is easy to see that Proposition \ref{aproximacion por curvas}  implies that the nodal separators $S$ and $\widetilde{S}$ are equisingular, which will finish the proof of Theorem \ref{Zariski}.

\section{Constructing a better topological equivalence}\label{better}
In this section we begin with the proof of Proposition \ref{aproximacion por curvas}. Concretely, this section is devoted to prove Proposition \ref{nice},  which permit us to construct, given a topological equivalence of nodal separators, another topological equivalence with ``nice'' properties. 

Let $p, \tilde{p}, p_j ,\tilde{p}_j, E_j, \widetilde{E}_j$ be as in Section~\ref{top-eq}.
Clearly,  it is sufficient to prove Proposition \ref{aproximacion por curvas} for $k\in\mathbb{N}$  large enough. Thus, from now on we assume  $k\in\mathbb{N}$  large enough such that:
\begin{enumerate}
\item $p_k$ is the intersection of $E_k$ with $E_l$ for some $l<k$;  
\item $\widetilde{p}_k$ is the intersection of $\widetilde{E}_k$ with $\widetilde{E}_{\widetilde{l}}$ for some $\widetilde{l}<k$;
\item $E_k$ and $E_l$ are the separatrices of the nodal foliation generating $S$;
\item $\widetilde{E}_k$ and $\widetilde{E}_{\widetilde{l}}$  are the separatrices of the nodal foliation generating $\widetilde{S}$.
\end{enumerate}
Denote by $M$ the complex surface obtained by performing the $k$ successively  blow ups at $p$, $p_1$, ...,  $p_{k-1}$.  Set $$E:=\bigcup_{j=1}^{k}E_j $$ and let  $$\pi\colon (M,E)\rightarrow(V,p) $$ be the natural map. In the same way define $\widetilde{M}$, $\widetilde{E}$  and the natural map  $$\widetilde{\pi}\colon (\widetilde{M},\widetilde{E})\rightarrow(\widetilde{V},\widetilde{p}). $$
Let $\mathfrak{h}:\mathfrak{U}\rightarrow\widetilde{\mathfrak{U}}$ be a topological equivalence between the nodal separators $S$ at $p$ and $\widetilde{S}$ at $\widetilde{p}$. Set $U=\pi^{-1}(\mathfrak{U})$,   $\widetilde{U}=\widetilde{\pi}^{-1}(\widetilde{\mathfrak{U}})$ and $$h=\widetilde{\pi}^{-1}\circ\mathfrak{h}\circ\pi:U\backslash E\rightarrow\widetilde{U}\backslash\widetilde{E}.$$
Clearly the following properties hold:

\begin{enumerate}
\item $h$ is a homeomorphism;
\item $h(\zeta)\rightarrow \widetilde{E}$ as $\zeta\rightarrow E$, that is, $\tilde{d}(h(\zeta),\widetilde{E})\rightarrow 0$ as $d(\zeta,E)\rightarrow 0$ for some metrics $d$ and $\tilde{d}$ on $M$ and $\widetilde{M}$ respectively;
\item $h(S\cap U\backslash\{p_k\})=\widetilde{S}\cap\widetilde{U}\backslash\{\widetilde{p}_k\}$;
\item the leaves of the Levi foliation of $S\cap U$ are mapped by $h$ onto the leaves of the Levi foliation of $\widetilde{S}\cap\widetilde{U}$.
\end{enumerate}
The following proposition is inmediate:
\begin{prop} If a map $h:U\backslash E\rightarrow\widetilde{U}\backslash\widetilde{E}$ satisfies the properties 1,2,3 and 4 above, then the map \begin{align*}&\mathfrak{h} \colon\mathfrak{U}\rightarrow\widetilde{\mathfrak{U}},\\
&\mathfrak{h} =\tilde{\pi}\circ h\circ \pi^{-1} \textrm{ on } U\backslash\{0\},\\
&\mathfrak{h}(0) =0
\end{align*} defines a topological equivalence between  the nodal separators $S$ at $p$ and $\widetilde{S}$ at $\widetilde{p}$.
\end{prop}
Thus, in order to construct a topological equivalence between the nodal separators $S$ at $p$ and $\widetilde{S}$ at $\widetilde{p}$ will be sufficient to construct a map $h$ satisfying the properties 1,2,3 and 4 above. Furthermore, if no confusion arise we can identify both maps $h$ and $\mathfrak{h}$. Then,    from now on it will be convenient to denote $h$ also by $\mathfrak{h}$.

\begin{prop} \label{nice}
Let $\mathfrak{h}\colon\mathfrak{U}\rightarrow\widetilde{\mathfrak{U}}$ be a topological equivalence between the nodal separators $S$  and $\widetilde{S}$. Then there exist:

\begin{enumerate}

\item another topological equivalence $\mathfrak{h}_1$ between $S$ and $\widetilde{S}$;
\item  local holomorphic coordinates $(x,y)$ at $p_k\in M$; 
\item local holomorphic coordinates $(\widetilde{x},\widetilde{y})$  at $\widetilde{p}_k\in\widetilde{M}$; 
\item a matrix $\begin{pmatrix}a&b\\ c&d\\ \end{pmatrix}$ in $ SL(2,\mathbb{Z})$; 
\item real irrational numbers $\lambda,\widetilde{\lambda}>0$; and
\item complex numbers $\mu_0,\nu_0\in\partial\mathbb{D}$

\end{enumerate}
such that:
\begin{enumerate}
\item $E_l=\{y=0\}$, $E_k= \{x=0\}$;
\item $\widetilde{E}_{\widetilde{l}}=\{\widetilde{y}=0\}$, $\widetilde{E}_k=\{\widetilde{x}=0\}$;
 \item $S=\{|y|=|x|^{\lambda}\}$;
 \item $\widetilde{S}=\{|\widetilde{y}|=|\widetilde{x}|^{\widetilde{\lambda}}\}$;
\item ${\widetilde{\lambda}}=\frac{c+d\lambda}{a+b\lambda};$

\item \label{item6}$\mathfrak{h}_1$ maps $\{|y|=|x|^{\lambda},  |x|\le 1\}$ onto
$\{|\widetilde{y}|=|\widetilde{x}|^{\lambda},  |\widetilde{x}|\le 1\}$ by the rule
$$ \mathfrak{h}_1 (t\eta,t^{\lambda}\xi)=(t\mu_0\eta^a\xi^b,t^{\widetilde{\lambda}}\nu_0\eta^c\xi^d); \, \eta,\xi\in\partial\mathbb{D},t\in[0,1].$$
\end{enumerate}
\end{prop}

\begin{rem}  Observe that the irrational numbers $\lambda,\tilde{\lambda}$ actually depend on the natural number $k$, which we have previously fixed taking into account the properties in the beginning of Section \ref{better} (see remark \ref{corner}). In order to prove Proposition \ref{aproximacion por curvas} we will approximate the nodal separators $S$ and $\widetilde{S}$ by curves of type $y=x^{\frac{m}{n}}$ and $\tilde{y}=\tilde{x}^{\frac{\tilde{m}}{\tilde{n}}}$  for rational numbers $\frac{m}{n}$ and $\frac{\tilde{m}}{\tilde{n}}$ close to $\lambda$ and $\tilde{\lambda}$ respectively. If we consider $k\in\mathbb{N}$ fixed,  a first option to obtain a satisfactory approximation to the infinitesimal behavior of $S$ is to take $\frac{m}{n}$ very close to $\lambda$.   Nevertheless, will be more convenient for us to think in the following different way: for each $k$ we can choose $\frac{m}{n}$ ``moderately'' close to $\lambda=\lambda(k)$, then $y=x^{\frac{m}{n}}$  will give an arbitrarily satisfactory approximation to the infinitesimal behavior of $S$  whenever we take $k$ large enough. The precise mean of the word ``moderately'' above will be established in Section \ref{proofAC}.

\end{rem}
We begin with the proof of Proposition \ref{nice}.

Let $ B'$ be a small diffeomorphic compact ball centered at $p\in V$ and contained in $\mathfrak{U}$. There exist holomorphic coordinates $(x,y)$ at $p_k$ such that the foliation associated to $S$ is given by the holomorphic vector field $x\frac{\partial}{\partial x}+\lambda y\frac{\partial}{\partial y}$ for some irrational number $\lambda>0$ . We can assume that the nodal separator $S$ is given by $\{|y|=|x|^{\lambda}\}$ at $p_k$. Take some $\epsilon>0$ and consider, for each $s\in[-1,1]$, the nodal separator $\mathcal{S}_s$ at $p\in V$ given in the infinitesimal coordinates $(x,y)$   by 

 $$\mathcal{S}_s=\{|y|=(1+s\epsilon)|x|^{\lambda}\}.$$ 

Set $S_s=\mathcal{S}_s\cap B'$ and $\Lambda=\displaystyle{\bigcup_{s\in[-1,1]} S_s}$.

$B'$ and $\epsilon$ can be taken such that  the following properties hold:

\begin{enumerate}

\item  $\partial  B'$ is transverse to each  $\mathcal{S}_s$;
\item in the infinitesimal coordinates $(x,y)$,  each intersection $T'_s=\partial B'\cap\mathcal{S}_{s}$ is given by $$\{|y|=(1+s\epsilon)|x|^{\lambda}, |x|=r'_s\},$$ for some $r'_s>0$;

\item the set $\Lambda\backslash \{p\}$ is diffeomorphic to $(S_0\backslash\{p\})\times[-1,1]$ in such way that

\begin{enumerate}

\item  $(S_s\backslash\{p\})\simeq (S_0\backslash\{p\})\times\{s\}$, and

\item the Levi foliation on    $(S_s\backslash\{p\})\simeq (S_0\backslash\{p\})\times\{s\}$ coincides with the Levi foliation on $(S_0\backslash\{p\})$.

\end{enumerate}

\end{enumerate}

It is easy to construct a continuous map $f$ on the closure of $ B'\backslash\Lambda$ with the following properties:

\begin{enumerate}

\item $f$ maps  $ B'\backslash\Lambda$ homeomorphically onto $ B'\backslash S_0$;

\item for $\sigma\in\{1,-1\}$, we have that  $f$ maps $S_{\sigma}\simeq S_0\times\{\sigma\}$ homeomorphically onto $S_0$ by the rule 
$(\zeta,\sigma)\mapsto \zeta$.

\end{enumerate}

Now, we proceed in an analogous way at $\widetilde{p}\in \widetilde{V}$:  let $\widetilde{ B}$ be a diffeomorphic compact ball centered at $\widetilde{p}\in \widetilde{V}$ and contained in $\widetilde{\mathfrak{U}}$. 
Let $(\widetilde{x},\widetilde{y})$ be holomorphic coordinates  at $\widetilde{p}_k$ such that the foliation associated to $\widetilde{S}$ is given by the holomorphic vector field $\widetilde{x}\frac{\partial}{\partial \widetilde{x}}+\widetilde{\lambda} \widetilde{y}\frac{\partial}{\partial \widetilde{y}}$. We can assume that the nodal separator $\widetilde{S}$ is given by $\{|\widetilde{y}|=|\widetilde{x}|^{\widetilde{\lambda}}\}$ at $\widetilde{p}_k$. Take some $\widetilde{\epsilon}>0$ and consider, for each $s\in[-1,1]$, the nodal separator $\widetilde{\mathcal{S}}_s$ at $\widetilde{p}\in \widetilde{V}$ given in the infinitesimal coordinates $(\widetilde{x},\widetilde{y})$   by 

 $$\widetilde{\mathcal{S}}_s=\{|\widetilde{y}|=(1+s\widetilde{\epsilon})|\widetilde{x}|^{\widetilde{\lambda}}\}.$$ 

Set $\widetilde{S}_s=\widetilde{\mathcal{S}}_s\cap \widetilde{B}$ and $\widetilde{\Lambda}=\displaystyle{\bigcup_{s\in[-1,1]} \widetilde{S}_s}$.

We can take  $\widetilde{B}$ and $\widetilde{\epsilon}$ such that   the following properties hold:

\begin{enumerate}

\item $\partial  \widetilde{B}$ is transverse to each $\widetilde{\mathcal{S}}_s$;

\item in the infinitesimal coordinates $(\widetilde{x},\widetilde{y})$,  each intersection $\widetilde{T}_s=\partial \widetilde{B}\cap\widetilde{\mathcal{S}}_s$ is given by $$\{|\widetilde{y}|=(1+s\widetilde{\epsilon})|\widetilde{x}|^{\widetilde{\lambda}}, |\widetilde{x}|=\widetilde{r}_s\},$$ for some $\widetilde{r}_s>0$:

\item the set $\widetilde{\Lambda}\backslash\{\tilde{p}\}$ is diffeomorphic to $(\widetilde{S}_0\backslash\{\tilde{p}\})\times[-1,1]$ in such way that

\begin{enumerate}

\item  $(\widetilde{S}_s\backslash\{\tilde{p}\})\simeq (\widetilde{S}_0\backslash\{\tilde{p}\})\times\{s\}$, and

\item the Levi foliation on    $(\widetilde{S}_s\backslash\{\tilde{p}\})\simeq (\widetilde{S}_0\backslash\{\tilde{p}\})\times\{s\}$ coincides with the Levi foliation on $(\widetilde{S}_0\backslash\{\tilde{p}\})$.

\end{enumerate}

\end{enumerate}

We construct a continuous map $\widetilde{f}$ on the closure of $ \widetilde{B}\backslash\widetilde{\Lambda}$ with the following properties:

\begin{enumerate}

\item $\widetilde{f}$ maps  $ \widetilde{B}\backslash\widetilde{\Lambda}$ homeomorphically onto $\widetilde{B}\backslash \widetilde{S}_0$;

\item for $\sigma\in\{1,-1\}$, we have that  $\widetilde{f}$ maps $\widetilde{S}_{\sigma}\simeq \widetilde{S}_0\times\{\sigma\}$ homeomorphically onto $\widetilde{S}_0$ by the rule 
$(\zeta,\sigma)\mapsto \zeta$.

\end{enumerate}

Clearly we can assume $B'$ small enough such that $\mathfrak{h}(B')$ is contained in the interior of $\widetilde{B}$. Then we can define the map $\mathfrak{h}_0=\tilde{f}^{-1}\circ\mathfrak{h}\circ {f}$ on $ B'\backslash\Lambda$. On  $$(\Lambda\backslash\{p\})\simeq (S_0\backslash\{p\})\times [-1,1]$$ define 
$$\mathfrak{h}_0(\zeta,s)=(\mathfrak{h}(\zeta),s)\in (\widetilde{S}_0\backslash\{\tilde{p}\})\times [-1,1]\simeq(\widetilde{\Lambda}\backslash\{\tilde{p}\})$$ and set $\mathfrak{h}_0(p)=\tilde{p}$.  It is easy to verify the following properties:

\begin{enumerate}

\item $\mathfrak{h}_0$ maps  $ B'$ homeomorphically into  $ \widetilde{B}$;
\item  if $s\in [-1,1]$, then $\mathfrak{h}_0$ maps $S_s$  homeomorphically into  $\widetilde{S}_s$ conjugating the Levi foliations.
\end{enumerate}

Let $B$  be  a diffeomorphic compact ball, centered at $p\in V$ and such that:

\begin{enumerate}

\item $B$ is contained $\mathfrak{U}$;

\item $B'$ is contained in the interior of $B$;

\item in the infinitesimal coordinates $(x,y)$, the intersection of each $\mathcal{S}_s$  with $\overline{B\backslash B'}$ is given by  
$$\{|y|=(1+s\epsilon)|x|^{\lambda}, r'_s\le |x|\le r_s\},$$ for some $r_s>r'_s$.

\end{enumerate}
Set: \begin{enumerate}
\item $C_s= \mathcal{S}_s\cap\overline{B\backslash B'};$ \item $\mathcal{C}=\displaystyle{\bigcup_{s\in[-1,1]} C_s};$
\item $\mathcal{T}' =\mathcal{C}\cap\partial B'=\{(x,y)\in \mathcal{C}\colon |x|=r'_s\};$
\item
$\mathcal{T} =\mathcal{C}\cap \partial B=\{(x,y)\in \mathcal{C}\colon |x|=r_s\}.$  
\end{enumerate}
Clearly $\mathcal{C}$ is foliated by the restrictions to $\mathcal{C}$ of the leaves of the foliations of each $\mathcal{S}_s$; in fact, this foliation on $\mathcal{C}$ is generated by the vector field $x\frac{\partial}{\partial x}+\lambda y\frac{\partial}{\partial y}$ in the infinitesimal coordinates $(x,y)$. 
Given $z=(x_z,y_z)\in\mathcal{T}'$, let $L_z$ be the leave in $\mathcal{C}$ passing through $z$.  Consider the path $\gamma_z:[0,1]\rightarrow L_z$, $\gamma_z(t)=(x(t),y(t))$ such that $\gamma_z(0)=z$ and $x(t)=(1-t+t\frac{r_s}{r_s\rq{}})x_z$. Clearly we have the following properties:

\begin{enumerate}

\item    $\gamma_z(1)\in\mathcal{T} $;
\item $z\mapsto \gamma_z(1)$ defines a homeomorphism between $\mathcal{T}' $ and $\mathcal{T} $;
\item $\gamma_z((0,1))$ is contained in the interior of $L_z$;
\item  the sets $I_z= \gamma_z([0,1])$, $z\in \mathcal{T}'$  define a 1-dimensional foliation of $\mathcal{C}$.

\end{enumerate}
Set:

 \begin{enumerate}

\item $\widetilde{C}_s= \widetilde{\mathcal{S}}_s\cap\overline{\widetilde{B}\backslash \mathfrak{h}_0 (B')};$

 \item $\widetilde{\mathcal{C}}=\displaystyle{\bigcup_{s\in[-1,1]}\widetilde{C}_s};$

\item
$T_s= \mathcal{S}_s\cap\partial{B};$

 \item $\widetilde{\mathcal{T}}=\displaystyle{\bigcup_{s\in[-1,1]} \widetilde{T}_s}=\widetilde{\mathcal{C}}\cap\widetilde{B}$.

\end{enumerate}

\begin{lem}\label{geometrico} There exist $\mu_s,\nu_s\in\partial\mathbb{D}$ depending continuously on $s\in[-1,1]$ and a matrix  $\begin{pmatrix}a&b\\ c&d\\ \end{pmatrix}$ in $ SL(2,\mathbb{Z})$  such that the homeomorphism
  $ \bar{\mathfrak{h}}:\mathcal{T}\rightarrow \widetilde{\mathcal{T}}$ defined by
$$ \bar{\mathfrak{h}}(r_s\eta,(1+s\epsilon){r_s}^{\lambda}\xi)=(\tilde{r}_s\mu_s\eta^a\xi^b,(1+s\tilde{\epsilon})\tilde{r}_s^{\tilde{\lambda}}\nu_s\eta^c\xi^d); \; \eta, \xi \in \partial\mathbb{D}, s\in[-1,1] $$
has the following properties:
\begin{enumerate}

\item  $ \bar{\mathfrak{h}}$ conjugates the foliations in $\mathcal{T}$ and $\widetilde{\mathcal{T}}$;
\item  for all $z\in\mathcal{T}'$, the points $\mathfrak{h}_0(z)$ and $ \bar{\mathfrak{h}}(\gamma_z(1))$ are contained in the same leaf of $\tilde{x}\frac{\partial}{\partial \tilde{x}}+\tilde{\lambda} \tilde{y}\frac{\partial}{\partial \tilde{y}}$.

\end{enumerate} 

\end{lem}

Before proceeding with the proof of Lemma \ref{geometrico}, we need to establish  the following dynamical lemma.
\begin{lem}\label{ass} Suppose that the maps  $H,A\colon\mathbb{R}^2\rightarrow\mathbb{R}^2$ satisfy the following hypothesis:
 \begin{enumerate}
\item $H$ is continuous and $A$ is a linear isomorphism;
\item $H(u+m,v+n)=H(u,v)+A(m,n),\textrm{ for all }(u,v)\in\mathbb{R}^2,\,(m,n)\in\mathbb{Z}^2$; 
\item there exist irrational numbers $\lambda,\tilde{\lambda}\in\mathbb{R}$ such that $H$ maps  leaves of the foliation $dv-\lambda du=0$ into leaves of the foliation $d{v}-\tilde{\lambda} d{u}=0$.
\end{enumerate} Then we have the following properties:
\begin{enumerate}
\item $\lambda$ and $\tilde{\lambda}$ are related by $$\widetilde{\lambda}=\frac{c+d\lambda}{a+b\lambda};$$
\item there exists a continuous function $\kappa:\mathbb{R}^2\rightarrow\mathbb{R}$ such that
\[H(u,v)=H(0,0)+A(u,v)+\kappa(u,v)\cdot(1,\tilde{\lambda}).\]
\end{enumerate}
\end{lem}

\begin{proof} \emph{(1)} Since $\lambda$ is irrational, given $k\in\mathbb{N}$ there exist $m_k,n_k\in\mathbb{Z}$ tending to infinite such that $$\delta_k:=m_k\lambda-n_k\rightarrow 0 \textrm{ as }k\rightarrow\infty.$$ Since $(m_k,\lambda m_k)$ and $(0,0)$ belong to the line $v-\lambda u=0$ and this line is mapped into a leaf of the foliation $d{v}-\tilde{\lambda} d{u}=0$, there exists $r_k\in\mathbb{R}$ such that 
\begin{equation}\label{dyn1}H(m_k,\lambda m_k)-H(0,0)=r_k(1,\tilde{\lambda}).\end{equation}
On the other hand we have 
\begin{align*}H(m_k,\lambda m_k)&=H(m_k,\delta_k+n_k)=H(0,\delta_k)+A(m_k,n_k)\\&=H(0,\delta_k)+A(m_k,m_k\lambda-\delta_k)\\&=H(0,\delta_k)+m_k A(1,\lambda)-A(0,\delta_k).\end{align*} From this and from Equation  \ref{dyn1} we obtain
\begin{equation}\label{dyn2}\nonumber
A(1,\lambda)=\frac{1}{m_k}H(0,0)-\frac{1}{m_k}H(0,\delta_k)+\frac{1}{m_k}A(0,\delta_k)+\frac{r_k}{m_k}(1,\tilde{\lambda}).
\end{equation} Then, if $k\rightarrow\infty$ in  last equation we deduce that $A(1,\lambda)=c(1,\tilde{\lambda})$ for some $c\in\mathbb{R}$. Then, since  $A$ is an isomorphism and therefore  $c\neq 0$, we conclude that   $$\widetilde{\lambda}=\frac{c+d\lambda}{a+b\lambda}.$$
\emph{(2)} Fix $(u_0,v_0)\in \mathbb{R}^2$. Given $k\in\mathbb{N}$ we now take $m_k,n_k\in\mathbb{Z}$ tending to infinite such that $$\delta_k:=m_k\lambda-n_k+v_0-u_0\lambda\rightarrow 0 \textrm{ as }k\rightarrow\infty.$$ 
Since $(u_0,v_0)$ and $(m_k,\delta_k+n_k)$ belong the line $v-\lambda u=v_0-\lambda u_0$, there exists $r_k\in\mathbb{R}$ such that 
\begin{align*}r_k(1,\tilde{\lambda})&=H(m_k,\delta_k+n_k)-H(u_0,v_0)\\&=H(0,\delta_k)+A(m_k,n_k)-H(u_0,v_0)\\&=H(0,\delta_k)-H(u_0,v_0)+A\big(u_0+m_k-u_0,m_k\lambda-\delta_k+v_0-u_0\lambda\big)\\&=H(0,\delta_k)-H(u_0,v_0)+A(u_0,v_0)-A(0,\delta_k)+ (m_k-u_0) A(1,\lambda)\\
&=H(0,\delta_k)-H(u_0,v_0)+A(u_0,v_0)-A(0,\delta_k)+ (m_k-u_0)c(1,\tilde{\lambda})\end{align*} and therefore we have 
$$H(u_0,v_0)-H(0,\delta_k)-A(u_0,v_0)=(m_kc-u_0c-r_k)(1,\tilde{\lambda})-A(0,\delta_k).$$
Then, if $k\rightarrow\infty$  we deduce that there exists $\kappa(u_0,v_0)\in\mathbb{R}$ such that $$H(u_0,v_0)-H(0,0)-A(u_0,v_0)=\kappa(u_0,v_0)(1,\tilde{\lambda}).$$ Clearly $\kappa$ is necessarily continuous, so the proof of the lemma is complete.
\end{proof}

\noindent\emph{Proof of Lemma \ref{geometrico}.}   Consider the real flow $\phi$ associated to the vector field  $\widetilde{x}\frac{\partial}{\partial \widetilde{x}}+\widetilde{\lambda} \widetilde{y}\frac{\partial}{\partial \widetilde{y}}$. Given $\zeta\in \widetilde{S}_s\backslash\{0\}$, let $\rho(\zeta)\in\widetilde{T}_s$ be the intersection intersection between $\widetilde{T}_s$   and the orbit of $\phi$  through $\zeta$.    Define the map $h_s:T_s\rightarrow\widetilde{T}_s$ as follows. Given $w\in{T}_s$, let $z\in\mathcal{T}'$ be such that $\gamma_z(1)=w$ and put $h_s(w)=\rho(\mathfrak{h}_0(z))$. Let $\mathcal{G}_s$ and $\widetilde{\mathcal{G}}_s$ be the one dimensional real foliations induced by the Levi foliations on $T_s$ and $\widetilde{T}_s$, respectively.   It is easy to verify the following properties:
\begin{enumerate}
\item $h_s$ maps leaves of $\mathcal{G}_s$ to leaves of $\widetilde{\mathcal{G}}_s$;
\item Although $h_s$ is not necessarily a homeomorphism, it induces an isomorphism $h_s^*:\pi_1(T_s)\rightarrow\pi_1(\widetilde{T}_s)$.

\end{enumerate}
Recall that $$T_s=\{(r_s\eta,(1+s\epsilon)r_s^{\lambda}{\xi})\colon\eta,\xi\in\partial\mathbb{D}\}$$ and 
$$\widetilde{T}_s=\{(\tilde{r}_s\eta,(1+s\tilde{\epsilon})\tilde{r}_s^{\widetilde{\lambda}}{\xi})\colon\eta,\xi\in\partial\mathbb{D}\}.$$ Consider the bases $\{\alpha_s, \beta_s\}$ of $\pi_1(T_s)$ and $\{\tilde{\alpha}_s,\tilde{\beta}_s\}$ of  $\pi_1(\widetilde{T}_s)$ given by the positively oriented loops \begin{eqnarray}\alpha_s&=&r_s\partial\mathbb{D}\times\{(1+s\epsilon)r_s^{\lambda}\};\\
 \beta_s&=&\{r_s\}\times (1+s\epsilon)r_s^{\lambda}\partial\mathbb{D};\\  \tilde{\alpha}_s&=&\tilde{r}_s\partial\mathbb{D}\times\{(1+s\tilde{\epsilon})\tilde{r}_s^{\tilde{\lambda}}\};\\ 
\tilde{\beta}_s&=&\{\tilde{r}_s\}\times (1+s\tilde{\epsilon})\tilde{r}_s^{\tilde{\lambda}}\partial\mathbb{D}.
\end{eqnarray} 
Let $A_s$ be the matrix in $SL(2,\mathbb{Z})$ representing the isomorphism $h_s^*$ respect to the bases above. In fact, it is easy to see that $A_s$ does not depend on $s\in[-1,1]$, so we have $$A_s=A=\begin{pmatrix}a&b\\ c&d\\ \end{pmatrix}\in \textrm{SL}(2,\mathbb{Z}).$$ Consider the coverings

$$(u,v)\in\mathbb{R}^2\mapsto (r_se^{2\pi iu},(1+s\epsilon)r_s^{\lambda}e^{2\pi iv})\in T_s;$$  
$$(\tilde{u},\tilde{v})\in\mathbb{R}^2\mapsto (\tilde{r}_s e^{2\pi i\tilde{u}}(1+s\tilde{\epsilon})\tilde{r}_s^{\tilde{\lambda}}e^{2\pi i\tilde{v}})\in \widetilde{T}_s.$$
 and let $H_s:\mathbb{R}^2\rightarrow\mathbb{R}^2$ be a lift of $h_s$. The pullbacks of $\mathcal{G}_s$ and $\widetilde{\mathcal{G}}_s$ in the planes $(u,v)$ and $ (\tilde{u},\tilde{v})$ define the foliations $dv-\lambda du=0$ and $d\tilde{v}-\tilde{\lambda} d\tilde{u}=0$, respectively.  It is easy to see the following:
\begin{enumerate}
\item $H_s(u+m,v+n)=H_s(u,v)+A(m,n),\textrm{ for all }(u,v)\in\mathbb{R}^2,\,(m,n)\in\mathbb{Z}^2$; 
\item $H_s$ maps leaves of    $dv-\lambda du=0$ into leaves of  $d\tilde{v}-\tilde{\lambda} d\tilde{u}=0$.
\end{enumerate}
By Lemma \ref{ass} there exists a continuous function $\kappa_s:\mathbb{R}^2\rightarrow\mathbb{R}$ such that
\begin{equation}\label{rigido}H_s(u,v)=H_s(0,0)+A(u,v)+\kappa_s(u,v)(1,\widetilde{\lambda})\end{equation} and we have the equality
 $\widetilde{\lambda}=\frac{c+d\lambda}{a+b\lambda}$. Consider the homeomorphism  $$\overline{H}_s(u,v)=H_s(0,0)+A(u,v)$$ and let $\bar{{h}}_s:T_s\rightarrow\widetilde{T}_s$ be the corresponding induced homeomorphism. Clearly $\overline{H}_s$ conjugates the foliations defined by     $dv-\lambda du=0$ and $d\tilde{v}-\tilde{\lambda} d\tilde{u}=0$, so $\bar{h}_s$ conjugates $\mathcal{G}_s$ with $\widetilde{\mathcal{G}}_s$ . Let $H_s(0,0)=(\mathfrak{u}_s,\mathfrak{v}_s)$ and define $\mu_s=e^{2\pi i \mathfrak{u}_s}$ and $ \nu_s=e^{2\pi i \mathfrak{v}_s}$. Then it is easy to see that 
 $$\bar{{h}}_s(r_s \eta,(1+s\epsilon)r_s^{\lambda}\xi)=(\tilde{r}_s\mu_s \eta^a\xi^b,(1+s\tilde{\epsilon})\tilde{r}_s^{\tilde{\lambda}}\nu_s\eta^c \xi^d),\textrm{ for all }\eta,\xi\in\partial\mathbb{D}.$$ 
Since $\bar{h}_s=\bar{\mathfrak{h}}|_{T_s}$ for all $s\in[-1,1]$, item 1 of Lemma \ref{geometrico} is easily obtained. From equation \ref{rigido} it is easy to see that, for each $W\in\mathbb{R}^2$, $s\in[-1,1]$, the points  $H_s(W)$ and $\overline{H}_s(W)$ are in the same leaf of $d\tilde{v}-\tilde{\lambda}d\tilde{u}=0$. Therefore,  for each $w\in T_s$,  $s\in[-1,1]$,  the points $h_s(w)$ and $\bar{h}_s(w)$ are in the same leaf of $\widetilde{\mathcal{G}}_s$. Since $h_s(w)=\rho(\mathfrak{h}_0(z))$ provided $w=\gamma_z(1)$, we have that $\rho(\mathfrak{h}_0(z))$ and   $\bar{h}_s(\gamma_z(1))$ are in the same leaf of $\widetilde{\mathcal{G}}_s$. Moreover, since $\rho$ preserves the leaves of    $\tilde{x}\frac{\partial}{\partial \tilde{x}}+\tilde{\lambda} \tilde{y}\frac{\partial}{\partial \tilde{y}}$, we have that $\mathfrak{h}_0(z)$ and   $\bar{h}_s(\gamma_z(1))$ are in the same leaf of $\tilde{x}\frac{\partial}{\partial \tilde{x}}+\tilde{\lambda} \tilde{y}\frac{\partial}{\partial \tilde{y}}$.
This proves item 2 of Lemma \ref{geometrico}.\qed 

Given $z\in\mathcal{T}'$, let $s_z\in[-1,1]$ be such that $z\in S_{s_z}$ and let $H_z$ be the leaf of the Levi foliation of $S_{s_z}$  containing $z$. We know that
 $S_{s_z}$ is mapped by $\mathfrak{h}_0$ into $\widetilde{S}_{s_z}$. Moreover,   $\mathfrak{h}_0(H_z)$ is contained  in the interior of  a leave $\widetilde{H}_z$ of the Levi foliation of $\widetilde{S}_{s_z}$. Let $\widetilde{L}_z$ be the closure of $\widetilde{H}_z\backslash \mathfrak{h}_0(H_z)$.  The interior of $\widetilde{L}_z$ is holomorphically equivalent to a disc, so we can consider the Poincar\'e metric in   the interior of $\widetilde{L}_z$. Let $\widetilde{\gamma}_z:\mathbb{R}\rightarrow \widetilde{L}_z$ be a geodesic such that 

\begin{enumerate}

\item[] $\widetilde{\gamma}_z(-\infty):=\displaystyle{\lim_{s\rightarrow -\infty}\widetilde{\gamma}_z(s)=\mathfrak{h}_0(z)}$
\item[] $\widetilde{\gamma}_z(+\infty):=\displaystyle{\lim_{s\rightarrow +\infty}\widetilde{\gamma}_z(s)= \bar{\mathfrak{h}}(z)}$ 
\end{enumerate} 
and set $\widetilde{I}_z=\widetilde{\gamma}_z(\mathbb{R}\cup\pm\infty)$.  We have the following properties:

\begin{enumerate}
\item although the parameterized geodesic $\widetilde{\gamma}_z$ is not uniquely defined, the set $\widetilde{I}_z$ is well defined and depends continuously on $z\in\mathcal{T}'$;
\item  the sets $\widetilde{I}_z$, $z\in \mathcal{T}'$  defines a partition of $\widetilde{\mathcal{C}}$.

\end{enumerate}

 In order to choose $\widetilde{\gamma}_z$
depending continuously on  $z\in\mathcal{T}'$ it suffices to define the value $\widetilde{\gamma}_z(0)$ depending continuously on 
$z\in\mathcal{T}'$. Observe the following facts:

\begin{enumerate}

\item$\widetilde{L}_z$ is diffeomorphic to a closed band and $\partial_1\widetilde{L}_z:=\widetilde{L}_z\cap \widetilde{B}$  is a component of its boundary;
\item Since $\partial \widetilde{B}$ is smooth, the boundary $\partial_1\widetilde{L}_z$ depends smoothly  on $z$. Observe that we can assume  $\partial\widetilde{B}$ to be real analytic near $\widetilde{\mathcal{T}}$.
\end{enumerate}
Then, it is not difficult to prove that, for each $z$, the euclidean length of $\widetilde{\gamma}_z$ is finite.   Moreover, it is easy to see that there is $\delta>0$ such that the euclidean length of $\gamma_z$ is greater than $\delta$ for all  $z\in\mathcal{T}'$. Then we can define $\widetilde{\gamma}_z(0)$ such that the euclidean length of $\widetilde{\gamma}_z|_{[0,+\infty)}$ is equal to $\delta$. It is not difficult to see that $\widetilde{\gamma}_z(0)$  depends continuously on $z\in\mathcal{T}'$.
Fix an increasing diffeomorphism $\phi:(0,1)\rightarrow\mathbb{R}$ and define the homeomorphism $h_z:I_z\rightarrow\widetilde{I}_z$ by
 
\begin{enumerate}

\item[] \hspace{2cm} $h_z(\gamma_z(s))=\widetilde{\gamma}_z(\phi(s))$, if $s\in(0,1)$;
\item[] \hspace{2
cm} $h_z(z)=\mathfrak{h}_0(z)$; 
\item[] \hspace{2
cm} $h_z(\gamma_z(1))= \bar{\mathfrak{h}}(\widetilde{\gamma}_z(1))$.

\end{enumerate} 
Now, we can extend the map $\mathfrak{h}_0$ to $\mathcal{C}$ by putting $\mathfrak{h}_0=h_z$ on $I_z$.  The extended $\mathfrak{h}_0$  has the following properties:

\begin{enumerate}
\item $\mathfrak{h}_0$ is a homeomorphism between $B'\cup \mathcal{C}$ and $\mathfrak{h}_0(B')\cup\widetilde{\mathcal{C}}$;
\item $\mathfrak{h}_0$ maps the nodal separator $$\{|y|=(1+s\epsilon)|x|^{\lambda}, |x|\le r_s\}$$   onto the nodal separator
 $$\{|\widetilde{y}|=(1+s\tilde{\epsilon})|\tilde{x}|^{\tilde{\lambda}}, |\tilde{x}|\le \tilde{r}_s\}.$$
\item $\mathfrak{h}_0$ maps  ${T}_s$ onto $\widetilde{{T}}_s$ by the rule
$$ \mathfrak{h}_0(r_s\eta,(1+s\epsilon){r_s}^{\lambda}\xi)=(\tilde{r}_s\mu_s\eta^a\xi^b,(1+s\tilde{\epsilon})\tilde{r}_s^{\tilde{\lambda}}\mu_s\eta^c\xi^d); \; \eta, \xi \in \partial\mathbb{D}, s\in[-1,1]. $$
\end{enumerate}
Put $\mathfrak{h}_0(x,y)=(\mathfrak{f}(x,y),\mathfrak{g}(x,y))$  and define $$ \mathfrak{h}_1\colon B'\cup \mathcal{C}\rightarrow \mathfrak{h}_0(B')\cup\widetilde{\mathcal{C}}$$ as follows:\\

 $\displaystyle{\mathfrak{h}_1(tx,t^{\lambda}y)=\bigg(|s|\mathfrak{f}\Big(\frac{t}{|s|}x,(\frac{t}{|s|})^{\lambda}y\Big),{|s|}^{\widetilde{\lambda}}\mathfrak{g}\Big(\frac{t}{|s|}x,(\frac{t}{|s|})^{\lambda}y\Big)\bigg),}$ \\
\vspace{-1mm}
{\flushright for $(x,y)\in T_s$, $0<t<|s|$, $0<|s|\le 1$;\\}
\vspace{5mm}
$\displaystyle{\mathfrak{h}_1(tx,t^{\lambda}y)=\Big(t\mathfrak{f}(x,y),{t}^{\tilde{\lambda}}\mathfrak{g}(x,y)\Big),}$\\ 
\vspace{-3mm}
{\flushright for $(x,y)\in T_s$, $|s|\le t\le 1$, $|s|\le 1$; and\\} 
\vspace{5mm}
$\mathfrak{h}_1=\mathfrak{h}_0$ otherwise.\\

\noindent It is easy to see that ${\mathfrak{h}_1}$  has the following properties:

\begin{enumerate}
\item ${\mathfrak{h}_1}$ is a homeomorphism between $B'\cup \mathcal{C}$ and $\mathfrak{h}_0(B')\cup\widetilde{\mathcal{C}}$;
\item ${\mathfrak{h}_1}$ maps the nodal separator $$\{|y|=|x|^{\lambda}, |x|\le r_0\}$$   onto the nodal separator
 $$\{|\widetilde{y}|=|\widetilde{x}|^{\widetilde{\lambda}}, |\widetilde{x}|\le \widetilde{r}_0\}$$
 by the rule
$$ {\mathfrak{h}_1}(tr_0\eta,t^{\lambda}{r_0}^{\lambda}\xi)=(t\tilde{r}_0\mu_0\eta^a\xi^b,t^{\tilde{\lambda}}\tilde{r}_0^{\tilde{\lambda}}\nu_0\eta^c\xi^d); \; \eta, \xi \in \partial\mathbb{D}, t\in[0,1]. $$
\end{enumerate}
Clearly we can extend  ${\mathfrak{h}_1}$ to a neighborhood of $$\{|y|=|x|^{\lambda}, |x|\le r_0\}.$$ Moreover, by a linear change of coordinates we can assume that $r_0=\widetilde{r}_0=1$, so the proof of Proposition \ref{nice} is complete.

\section{Proof of Proposition \ref{aproximacion por curvas}}\label{proofAC}

By simplicity, we can assume that  $\mathfrak{h}:\mathfrak{U}\rightarrow\widetilde{\mathfrak{U}}$ satisfies the properties 1 to 6 in Proposition \ref{nice}. 

Consider $(\alpha,\beta)\in\mathbb{R}^+\times\mathbb{R}^+$ fixed. It is easy to see that the family of real curves $ (t^{\alpha}\eta,t^{\beta}\xi)$, $t\in[0,1]$ indexed by  $(\eta,\xi)\in\partial (\mathbb{D}\times \mathbb{D})$ defines a 1-dimensional foliation on $\overline{\mathbb{D}}\times \overline{\mathbb{D}}$ topologically equivalent to the standard real radial foliation. In particular, any $(x,y)\in (\overline{\mathbb{D}}\times \overline{\mathbb{D}})\backslash\{0\}$ can be expressed in a unique way as $$(x,y)= (t^{\alpha}\eta,t^{\beta}\xi)$$ for a some $(\eta,\xi)\in\partial (\mathbb{D}\times \mathbb{D})$, $t\in(0,1]$. We will need the following  lemma.
\begin{lem}\label{mn} Given $m,n\in\mathbb{N}$ define
$f:\overline{\mathbb{D}}\times \overline{\mathbb{D}}\rightarrow\overline{\mathbb{D}}\times\overline{\mathbb{D}}$ as follows. Firstly, define $f(0,0)=(0,0)$.  Secondly, if $(x,y)\neq 0$, from the considerations above we have  $$(x,y)= (t^{m}\eta,t^{n}\xi)$$ for a some $(\eta,\xi)\in\partial (\mathbb{D}\times \mathbb{D})$, $t\in(0,1]$, so we can define 
$$f(x,y)=(t\eta,t^{\lambda}\xi).$$  Then we have the following properties:
\begin{enumerate}
\item $f$ is a homeomorphism;
\item $f$ maps $\{|y|=|x|^{\frac{n}{m}}, |x|\le 1\}$ onto $\{|y|=|x|^{\lambda},|x|\le 1\};$
\item \label{itemid} $f=\emph{id}$ on $\partial(\mathbb{D}\times\mathbb{D})$.
\end{enumerate}
\end{lem}
\begin{rem} Observe that for any $(\eta,\xi)\in\partial (\mathbb{D}\times \mathbb{D})$ we can write  $$f(t^{m}\eta,t^{n}\xi)=(t\eta,t^{\lambda}\xi)$$  even if $t=0$. So we can consider that $f$ is defined by the unique expression $$f(t^{m}\eta,t^{n}\xi)=(t\eta,t^{\lambda}\xi)$$  for any $(\eta,\xi)\in\partial (\mathbb{D}\times \mathbb{D})$, $t\in[0,1]$.  
\end{rem}
\noindent\emph{Proof of Lemma \ref{mn}.}

\emph{(1)} For the first assertion it is sufficient to see that $f$ defines a topological equivalence between the topologically radial foliations defined by the pairs $(m,n)$  and $(1,\lambda)$. 

\emph{(2)} Given $(x,y)$ such that $|y|=|x|^{\frac{n}{m}}$, $|x|\le 1$, we easily see that $(x,y)=(t^m\eta,t^n\xi)$ with $|\eta|=|\xi|=1$, $t\in[0,1]$. Then $$(x',y'):=f(x,y)=(t\eta,t^{\lambda}\xi)$$ clearly satisfies $|y'|=|x'|^{\lambda}$. On the other hand, any $(x',y')$ such that  $|y'|=|x'|^{\lambda}$, $|x'|\le 1$ can be expressed as $(x',y')=(t\eta,t^{\lambda}\xi)$ with $|\eta|=|\xi|=1$, $t\in[0,1]$. Then  
$(x',y')=f(t^m\eta,t^n\xi)$, where $(x,y)=(t^m\eta,t^n\xi)$ obviously satisfies  $|y|=|x|^{\frac{n}{m}}$, $|x|\le 1$. This proves the second assertion. 

\emph{(3)} For the third assertion it is sufficient to see that $(x,y)=(t^m\eta,t^n\xi)\in\partial(\mathbb{D}\times\mathbb{D})$ implies $t=1$, so $f(x,y)=f(\eta,\xi)=(\eta,\xi)$.
\qed \\

We see from Proposition \ref{nice} that $$\frac{c+d\lambda}{a+b\lambda}>0;$$
hence
\begin{enumerate}
\item  ${a+b\lambda}>0$  and ${c+d\lambda}>0$; or
\item    ${a+b\lambda}<0$  and ${c+d\lambda}<0$.
\end{enumerate}
Take  $m,n\in\mathbb{N}$ with $n/m$ irreducible and close enough to $\lambda$ such that:
\begin{enumerate}
\item  ${am+bn}>0$  and ${cm+dn}>0$; or
\item  $am+bn<0$  and $cm+dn<0$.
\end{enumerate}
Let $f$ be as in Lemma \ref{mn}. Then $f$ defines a homeomorphism of the neighborhood $ \overline{\mathbb{D}}\times \overline{\mathbb{D}}$ of $p_k\in M$ with itself. If we put $f=\textrm{id}$ on $M \backslash \overline{\mathbb{D}}\times \overline{\mathbb{D}}$, from item \ref{itemid} of Lemma \ref{mn} we have the following properties:

\begin{enumerate}
\item $f$ is a homeomorphism of $M$ with itself;
\item $f(E)=E$;
\item $f$ maps $\{|y|=|x|^{\frac{n}{m}}, |x|\le 1\}$ onto $\{|y|=|x|^{\lambda},|x|\le 1\}$ by the rule
$$f(t^m\eta,t^n\xi)=(t\eta,t^{\lambda}\xi);\; (\eta,\xi)\in\partial(\mathbb{D}\times \mathbb{D}),t\in[0,1].$$

\end{enumerate}

If we set \begin{eqnarray}\nonumber\widetilde{m}&=&|am+bn|,\\ \nonumber \widetilde{n}&=&|cm+dn|
\end{eqnarray}  and apply Lemma \ref{mn} to a neighborhood of $\widetilde{p}_k$ in $\widetilde{M}$, we can construct as above a map $\widetilde{f}$ such that:

\begin{enumerate}

\item $\widetilde{f}$ is a homeomorphism of $\widetilde{M}$ with itself;
\item $\widetilde{f}(\widetilde{E})=\widetilde{E}$;
\item $\widetilde{f}$ maps $\{|\widetilde{y}|=|\widetilde{x}|^{\frac{\widetilde{n}}{\widetilde{m}}}, |\widetilde{x}|\le 1\}$ onto $\{|\widetilde{y}|=|\widetilde{x}|^{\widetilde{\lambda}},|\widetilde{x}|\le 1\}$  by the rule
$$\widetilde{f}({t}^{\widetilde{m}}\eta,t^{\widetilde{n}}\xi)=(t\eta,t^{\widetilde{{\lambda}}}\xi);\; (\eta,\xi)\in\partial(\mathbb{D}\times\mathbb{D}),t\in[0,1].$$

\end{enumerate}

If we consider the map $\mathfrak{h}_1:=\widetilde{f}^{-1}\circ \mathfrak{h}\circ f$,  clearly we have the following properties:
\begin{enumerate}
\item  $\mathfrak{h}_1$ maps the complement of $E$ in a neighborhood of $M$ onto the complement of  $\widetilde{E}$ in a neighborhood of $\widetilde{M}$;
\item $\mathfrak{h}(\zeta)\rightarrow \widetilde{E}$ as $\zeta\rightarrow E$.
\end{enumerate}
Moreover, from item 6 of Proposition \ref{nice}   we obtain an explicit expression of $\mathfrak{h}_1$ on $\{|y|=|x|^{\frac{n}{m}}, |x|\le 1\}$ as follows. If $(x,y)$  belongs to $\{|y|=|x|^{\frac{n}{m}}, |x|\le 1\}$, as we have seen in the proof of Lemma \ref{mn} we have
  $(x,y)=(t^m\eta,t^n\xi)$ with $|\eta|=|\xi|=1$,  $t\in[0,1]$ and  therefore:

\begin{eqnarray}\nonumber \mathfrak{ h}_1(x,y)&=&\widetilde{f}^{-1}\circ \mathfrak{ h}\circ f(t^m\eta,t^n\xi)
=\widetilde{f}^{-1}\circ\mathfrak{h}(t\eta,t^{\lambda}\xi)\\  \nonumber
&=&\widetilde{f}^{-1}(t\mu_0\eta^a\xi^b,t^{\widetilde{\lambda}}\nu_0\eta^c\xi^d)=
(t^{\widetilde{m}}\mu_0\eta^a\xi^b,t^{\widetilde{n}}\nu_0\eta^c\xi^d) \\ \nonumber
&=& (t^{|am+bn|}\mu_0\eta^a\xi^b,t^{|cm+dn|}\nu_0\eta^c\xi^d).
\end{eqnarray}
Here we have to cases.
In the first case we have $|am+bn|=am+bn$ and $|cm+dn|=cm+dn$ and therefore:

\begin{eqnarray}\nonumber \mathfrak{h}_1(x,y)&=& (t^{am+bn}\mu_0\eta^a\xi^b,t^{cm+dn}\nu_0\eta^c\xi^d)\\ \nonumber
&=& (\mu_0(t^m\eta)^a(t^n\xi)^b, \nu_0(t^m\eta)^c(t^n\xi)^d)\\ \nonumber
&=&(\mu_0 x^ay^b,\nu_0 x^cy^d).
\end{eqnarray}
In the other case we have

\begin{eqnarray}\nonumber \mathfrak{h}_1(x,y)&=& (t^{-am-bn}\mu_0\eta^a\xi^b,t^{-cm-dn}\nu_0\eta^c\xi^d)\\ \nonumber
&=& (\mu_0(t^m\eta^{-1})^{-a}(t^n\xi^{-1})^{-b},\nu_0(t^m\eta^{-1})^{-c}(t^n\xi^{-1})^{-d})\\ \nonumber
&=&(\mu_0{\overline{x}}^{-a}\overline{y}^{-b},\nu_0\overline{x}^{-c}\overline{y}^{-d}).
\end{eqnarray}
In any case, it is easy to see that $\mathfrak{ h}_1$ maps the curve $$\{(z^m,z^n):|z|<1\}\textrm{ at $p_k\in M$ }$$
to the curve $$\{(\mu_0 z^{\tilde{m}},\nu_0z^{\tilde{n}}):|z|<1\} \textrm{ at $\tilde{p}_k\in\widetilde{M}$ }.$$
Clearly these curves define two curves $\mathfrak{C}$ at $p$  and $\widetilde{\mathfrak{ C}}$ at $\widetilde{p}$ satisfying the properties 1, 2 and 3 of Proposition \ref{aproximacion por curvas}.

\section{Topological invariance of the eigenvalue}

Let $V$ and $\widetilde{V}$ be smooth complex surfaces and let $S$ and $\widetilde{S}$ be nodal separators at  $p\in V$  and at $\tilde{p}\in \widetilde{V}$, respectively.  We know that, after a finite sequence of blow ups at $p$, the nodal separator $S$ is generated by a nodal foliation with an irrational positive eigenvalue $\lambda$. Clearly this eigenvalue depends on the number of iterated blow ups realized at $p$, but  next theorem shows that, taking into consideration this number of blow ups, the eigenvalue is a topological invariant of the nodal separator. Moreover,  next theorem also show that there are only to possibilities for the map induced  by a topological equivalence between $S\backslash\{p\}$ and $\widetilde{S}\backslash\{\tilde{p}\}$  at homology level.
 
 \begin{prop} \label{best}
Let $\mathfrak{h}\colon\mathfrak{U}\rightarrow\widetilde{\mathfrak{U}}$ be a topological equivalence between the nodal separators $S$  and $\widetilde{S}$.  Let $ p_j ,\tilde{p}_j, E_j, \widetilde{E}_j$ $(j\in\mathbb{N}$) be as in Section~\ref{top-eq}. Let $k\in\mathbb{N}$ be such that 
\begin{enumerate}
\item $p_k$ is the intersection of $E_k$ with $E_l$ for some $l<k$;
\item $\tilde{p}_k$ is the intersection of $\widetilde{E}_k$ with $\widetilde{E}_{\tilde{l}}$ for some $\tilde{l}<k$;
\item $S$ at $p$ is generated by a nodal foliation whose separatrices are contained in $E_k$ and $E_l$; 
\item  $\widetilde{S}$ at $\tilde{p}$ is generated by a nodal foliation whose separatrices are contained in $\widetilde{E}_k$ and $\widetilde{E}_{\tilde{l}}$.
\end{enumerate} Let $(x,y)$ and  $(\tilde{x},\tilde{y})$ be holomorphic coordinates at $p_k$ and at  $\tilde{p}_k$, respectively,  such that  
\begin{enumerate}
\item   $p\simeq(0,0)$, $\tilde{p}\simeq (0,0)$; 
\item $E_l=\{y=0\}$, $E_k= \{x=0\}$,  $\widetilde{E}_{\tilde{l}}=\{\widetilde{y}=0\}$, $\widetilde{E}_k=\{\widetilde{x}=0\}$;
\item  $S$ at  $p$  is given by $\{|y|=|x|^{\lambda}\}$ for some irrational number $\lambda>0$;
\item  $\widetilde{S}$ at $\tilde{p}$  is given by $\{|\tilde{y}|=|\tilde{x}|^{\tilde{\lambda}}\}$ for some irrational number $\tilde{\lambda}>0$.
\end{enumerate} Let  $\mathfrak{h}^*$ be the map  from $H_1(S\backslash\{p_k\})$ to $H_1(\widetilde{S}\backslash\{\tilde{p}_k\})$ induced by $\mathfrak{h}$ at homology level. Clearly these groups can be naturally identified if we think $(x,y)\simeq(\tilde{x},\tilde{y})$, so we can think that $\mathfrak{h}^*$ is an isomorphism of $\mathbb{Z}^2$. Then, we have the following properties:
\begin{enumerate}
\item \label{it1} $\tilde{l}=l$;
\item $\tilde{\lambda}=\lambda$;
\item the map $\mathfrak{h}^*$ is the identity or the inversion isomorphism according to $\mathfrak{h}$ preserves or reverses the natural orientations of Levi foliations leaves.
\end{enumerate}

\end{prop}

\begin{proof} Item \ref{it1} follows directly from the equisingularity of $S$ and $\widetilde{S}$.   We return to the ideas and notations of Section \ref{proofAC}. From the final of the proof of  Proposition \ref{aproximacion por curvas} we deduce  that the curves $$\{(z^m,z^n):|z|<1\}\textrm{ at $p_k\in M$ }$$
and $$\{(\mu_0 z^{\tilde{m}},\nu_0z^{\tilde{n}}):|z|<1\} \textrm{ at $\tilde{p}_k\in\widetilde{M}$ }$$ are equisingular.  But this can happen only if we have $\frac{\tilde{n}}{\tilde{m}}=\frac{n}{m}$ or $\frac{\tilde{n}}{\tilde{m}}=\frac{m}{n}$, so 
\begin{align*}\frac{cm+dn}{am+bn}=\frac{n}{m}\textrm{ or }&\frac{cm+dn}{am+bn}=\frac{m}{n},\\cm^2-bn^2+(d-a)mn=0 \textrm{ or }& dn^2-am^2+(c-b)mn=0.\end{align*} Since  $\frac{n}{m}$ is any irreducible fraction close enough to $\lambda$, we conclude that $c=b=0,a=d$ or $a=d=0, c=b$. Thus,  $$\tilde{\lambda}=\frac{c+d\lambda}{a+b\lambda}\in\{\lambda,\frac{1}{{\lambda}}\}.$$ By the equisingularity of $S$ and $\widetilde{S}$ we have that $S$ is tangent to $E_l$ if and only if $\widetilde{S}$ is tangent to $\widetilde{E}_l$, so we  have $\lambda>1$ if and only if $\tilde{\lambda}>1$. Therefore  $\tilde{\lambda}=\lambda$ and  \begin{align*} \begin{pmatrix}a&b\\ c&d\\ \end{pmatrix}&=\pm\textrm{id}.\end{align*}  From the construction of the  map $\mathfrak{h}_1$ given by Proposition \ref{nice} it is easy to see that 
\begin{enumerate}
\item $\mathfrak{h}_1$ induces the same map $\mathfrak{h}^*$;
\item $\mathfrak{h}_1$ preserves the orientation of  Levi leaves if and only if $\mathfrak{h}$ do.
\end{enumerate} From the proof of Lemma \ref{geometrico} we see that the map   $\mathfrak{h}^*$ is given by the matrix $\begin{pmatrix}a&b\\ c&d\\ \end{pmatrix}$, so $\mathfrak{h}^*=\pm\textrm{id}$. Finally it is easy to see from item \ref{item6} of Proposition \ref{nice} that $\mathfrak{h}_1$ preserves the orientation of the leaves if  $\begin{pmatrix}a&b\\ c&d\\ \end{pmatrix}=\textrm{id}$ and reverses them if $\begin{pmatrix}a&b\\ c&d\\ \end{pmatrix}=-\textrm{id}$.
\end{proof}

\section{Topological equivalence of holomorphic foliations and invariance of nodal separators}\label{fol1}

Let  $\mathcal{F}$ and $\widetilde{\mathcal{F}}$ be holomorphic foliations with isolated singularities at $0\in\mathbb{C}^2$. Suppose that
$\mathcal{F}$ and $\widetilde{\mathcal{F}}$ are {topologically equivalent} (at $0\in\mathbb{C}^2$), that is,  there is an orientation preserving homeomorphism $\mathfrak{h}:{\mathfrak{U}}\rightarrow\widetilde{{\mathfrak{U}}}$, $\mathfrak{h}(0)=0$  between neighborhoods of $0\in\mathbb{C}^2$,  taking leaves of  $\mathcal{F}$ to leaves of $\widetilde{\mathcal{F}}$. Such a homeomorphism is called a topological equivalence  between $\mathcal{F}$ and $\widetilde{\mathcal{F}}$.

\begin{thm}\label{invariancia}  Let $\mathfrak{h}:{\mathfrak{U}}\rightarrow\tilde{{\mathfrak{U}}}$, $\mathfrak{h}(0)=0$  be a topological equivalence between  $\mathcal{F}$ and $\widetilde{\mathcal{F}}$.  Let $S$ be a nodal separator of $\mathcal{F}$ at $0\in\mathbb{C}^{2}$. Then $\mathfrak{h}(S)$ is a nodal separator of $\widetilde{\mathcal{F}}$ at $0\in\mathbb{C}^2$.
\end{thm}
\begin{proof} From Theorem \ref{cr} we deduce that $\mathfrak{h}(S)$ contains a nodal separator $\widetilde{S}$ of $\widetilde{\mathcal{F}}$ at $0\in\mathbb{C}^2$.There are infinitesimal coordinates $(\widetilde{x},\widetilde{y})$  such that $$\widetilde{S}=\{|\widetilde{y}|=|\widetilde{x}|^{\widetilde{\lambda}}: |\widetilde{x}|< 1\}.$$  It is sufficient to prove that there is some neighborhood $\widetilde{\mathfrak{U}}_0$ of $0\in\mathbb{C}^2$ such that $\mathfrak{h}(S)\cap\widetilde{\mathfrak{U}}_0$ is contained in $\widetilde{S}$. Take a neighborhood $\widetilde{\mathfrak{U}}_0$ of $0\in\mathbb{C}^2$ with the following properties:

\begin{enumerate}

\item $\widetilde{S}_0:=\widetilde{S}\cap\widetilde{\mathfrak{U}}_0\subset\{|\widetilde{y}|=|\widetilde{x}|^{\widetilde{\lambda}}:|\widetilde{x}|<1/2\};$
\item in infinitesimal coordinates $(x,y)$, we have
$$S_0:=S\cap\mathfrak{h}^{-1}(\widetilde{\mathfrak{U}}_0)=\{|y|=|x|^\lambda:|x|<1\}.$$

\end{enumerate}
Take a point $p\in S_0$ such that $\widetilde{p}:=\mathfrak{h}({p})$ is contained in $\widetilde{S}_0$. Let $L$ be the leaf of $\mathcal{F}|_{S_0}$ through $p$. Since $\mathfrak{h}(S)\cap\widetilde{\mathfrak{U}}_0=\mathfrak{h}(S_0)$ and $L$ is dense in $S_0$, it is sufficient  to prove that $\mathfrak{h}(L)$ is contained in $\widetilde{S}$. Let $\widetilde{L}$ be the leaf of $\widetilde{\mathcal{F}}|_{\widetilde{S}_0}$ through $\widetilde{p}$. By item 1 above we deduce that $\widetilde{L}$ is also the leaf of $\widetilde{\mathcal{F}}|_{\widetilde{\mathfrak{U}}_0}$  through $\widetilde{p}$.  Since $\mathfrak{h}(L)$ is contained in $\widetilde{\mathfrak{U}}_0$, then it is contained in  the leaf of $\widetilde{\mathcal{F}}|_{\widetilde{\mathfrak{U}}_0}$ through $\widetilde{p}$, so  $\mathfrak{h}(L)\subset \widetilde{L}$. Therefore   $\mathfrak{h}(L)\subset\widetilde{S}$.
\end{proof}

\begin{thm}\label{parejas de nodos}  Let $\mathfrak{h}:{\mathfrak{U}}\rightarrow\widetilde{{\mathfrak{U}}}$, $\mathfrak{h}(0)=0$  be a topological equivalence between  $\mathcal{F}$ and $\widetilde{\mathcal{F}}$.  Let $S_1$ and $S_2$ be  nodal separators of $\mathcal{F}$ at $0\in\mathbb{C}^{2}$  issuing from the same node in the resolution of $\mathcal{F}$.   Then $\mathfrak{h}(S_1)$ and $\mathfrak{h}(S_2)$ are nodal separators issuing from the same node in the resolution of $\widetilde{\mathcal{F}}$.
\end{thm}
\begin{proof}

Let $S$ be any nodal separator of $\mathcal{F}$. Denote by $\mathfrak{n}(S)$ the node in the resolution of $\widetilde{\mathcal{F}}$ associated to the nodal separator  $\mathfrak{h}(S)$. It is easy to see   that, if $S'$ is a nodal separator close enough to $S$, then $\mathfrak{h}(S')$ is close to $\mathfrak{h}(S)$  and  contains a nodal separator also issuing  from $\mathfrak{n}(S)$. Therefore, from Theorem \ref{invariancia} we deduce that $\mathfrak{n}(S')=\mathfrak{n}(S)$.  Thus, the map $\mathfrak{n}$ is locally constant and the theorem follows by an argument of connectedness.
\end{proof}
 \noindent\emph{Proof of Theorem \ref{bijection}.} It is a direct consequence of Theorem \ref{parejas de nodos}. 

\noindent\emph{Proof of Theorem \ref{nodalext}.}
By \cite{rosas3} there exists a topological equivalence $\mathfrak{h}$ between $\mathcal{F}$ and $\widetilde{\mathcal{F}}$ which, after resolution, extends as a homeomorphism to a neighborhood of each linearizable or resonant singularity which is not a corner. We denote by $\mathcal{E}$ and $\widetilde{\mathcal{E}}$ the exceptional divisors in the resolutions of $\mathcal{F}$ and $\widetilde{\mathcal{F}}$, respectively. We use the same notation $\mathcal{F}$ for the foliation at $(\mathbb{C}^2,0)$ an its strict transform by the resolution map.  Let $p$ be a nodal corner point of $\mathcal{F}$ and let $\tilde{p}$ its corresponding nodal point in $\widetilde{\mathcal{F}}$ according to Theorem \ref{bijection}. By Theorem \ref{Zariski}, the nodal separators at $p$ and at $\tilde{p}$ are equisingular, so $p$ and $\tilde{p}$ have the same eigenvalue $\lambda>0$. There are holomorphic coordinates $(x,y)$ at $p$  and $(\tilde{x}, \tilde{y})$  at $\tilde{p}$ with the following properties:
\begin{enumerate}

\item $p\simeq (0,0)$, $\tilde{p}\simeq (0,0)$;
\item $\{x=0\}$ and $\{y=0\}$ are contained in different components of $\mathcal{E}$;
\item $\{\tilde{x}=0\}$ and $\{\tilde{y}=0\}$ are contained in different components of $\widetilde{\mathcal{E}}$;
\item $\mathcal{F}$ is defined by the vector field $x\frac{\partial}{\partial x} +\lambda y\frac{\partial}{\partial y}$;
\item $\widetilde{\mathcal{F}}$ is defined by the vector field $\tilde{x}\frac{\partial}{\partial \tilde{x}} +\lambda \tilde{y}\frac{\partial}{\partial \tilde{y}}$.

\end{enumerate} We denote by $\mathcal{E}_1$  and $\mathcal{E}_2$ the connected components of $\mathcal{E}$ containing  $\{y=0\}$ and $\{x=0\}$, respectively. Analogously define $\widetilde{\mathcal{E}}_1$  and $\widetilde{\mathcal{E}}_2$.
We will use the ideas used in \cite{rosas3} to construct the topological equivalence near a nodal non corner point.  We will think for a  moment that $p$ is not a corner: think that $\{x=0\}$ is a separatrix and that the exceptional divisor is reduced to $\mathcal{E}_1$. The fact that $\{x=0\}$ is a separatrix mapped into $\{\tilde{x}=0\}$ is only used to remove the homological obstruction to the extension of  $\mathfrak{h}$, as we explain in the sequel. Set $$T=\{0<|x|\le\varepsilon, 0<|y|\le\varepsilon\}$$ and $$\widetilde{T}=\{0<|\tilde{x}|\le\varepsilon, 0<|\tilde{y}|\le\varepsilon\}$$ for some $\varepsilon>0$. The map  $\mathfrak{h}$ induces an isomorphism $\mathfrak{h}^*$ between $H_1 (T)$ and $H_1(\widetilde{T})$. In a natural way we can think that  $H_1 (T)=H_1(\widetilde{T})$.  Then, the fact that $\{x=0\}$ is a separatrix is used in \cite{rosas3} to prove that the isomorphism $\mathfrak{h}^*$ is the identity or the inversion isomorphism according to $\mathfrak{h}$ preserves or reverses the natural orientation of the leaves. In our case we already have this property, by Proposition \ref{best}. Then, given $\epsilon>0$,  as in  \cite[Theorem 7 and Section 7]{rosas3} we find some numbers $a_1,b_1,\tilde{a}_1,\tilde{b}_1\in (0,\epsilon)$ and construct a homeomorphism $\mathfrak{h}_1$ with the following properties:
\begin{enumerate}
\item $\mathfrak{h}_1$ is defined on $$V_1=W_1\backslash\Big(\{|x|< a_1, |y|< b_1\}\cup\mathcal{E}_1\Big),$$ where $W_1$ is a neighborhood of $\mathcal{E}_1$;
\item  $\mathfrak{h}_1$ maps $V_1$  onto $$\widetilde{W}_1\backslash \Big(\{|\tilde{x}|< \tilde{a}_1, |\tilde{y}|< \tilde{b}_1\})\cup \widetilde{\mathcal{E}}_1\Big),$$ where $\widetilde{W}_1$ is a neighborhood of $\widetilde{\mathcal{E}}_1$;
\item $\mathfrak{h}_1$ maps leaves of $\mathcal{F}$ to leaves of $\widetilde{\mathcal{F}}$;
\item  $\mathfrak{h}_1(\zeta)$ tends to $\widetilde{\mathcal{E}}_1$ as $\zeta$ tends to ${\mathcal{E}}_1$;
\item $\mathfrak{h}_1$ maps the set $$R_1=\{|x|=a_1, 0<|y|\le b_1\}$$ onto the set $$\widetilde{R}_1=\{|\tilde{x}|=\tilde{a}_1, 0<|\tilde{y}|\le \tilde{b}_1\}$$ conjugating the one dimensional foliations induced by $\mathcal{F}$ and $\widetilde{\mathcal{F}}$  on $R_1$ and $\widetilde{R}_1$;
\item $\mathfrak{h}_1$ maps each punctured disc $$\{x=u, 0<|y|\le b_1\}, |u|=a_1$$ onto a punctured disc  $$\{\tilde{x}=\tilde{u}, 0<|\tilde{y}|\le \tilde{b}_1\}, |\tilde{u}|=\tilde{a}_1,$$ so $\mathfrak{h}_1$ extends to $\overline{R_1}$;
\item close to the divisor and outside $$\{|x|\le\epsilon,|y|\le\epsilon\}\cup\mathfrak{h}^{-1}\big(\{|\tilde{x}|\le\epsilon,|\tilde{y}|\le\epsilon\}\big)$$ we have $\mathfrak{h}_1=\mathfrak{h}$;
\item $\mathfrak{h}_1$ induces the same map $\mathfrak{h}^*$.
\end{enumerate}  

In the same way, we find numbers $a_1,b_1,\tilde{a}_1,\tilde{b}_1\in (0,\epsilon)$ and construct a homeomorphism $\mathfrak{h}_2$ with the following properties:

\begin{enumerate}
\item $\mathfrak{h}_2$ is defined on $$V_2=W_ 2\backslash\Big(\{|x|< a_ 2, |y|< b_ 2\}\cup\mathcal{E}_2\Big),$$ where $W_ 2$ is a neighborhood of $\mathcal{E}_ 2$;
\item  $\mathfrak{h}_ 2$ maps $V_ 2$  onto $$\widetilde{W}_ 2\backslash\Big(\{|\tilde{x}|< \tilde{a}_ 2, |\tilde{y}|< \tilde{b}_ 2\})\cup \widetilde{\mathcal{E}}_2\Big),$$ where $\widetilde{W}_ 2$ is a neighborhood of $\widetilde{\mathcal{E}}_ 2$;
\item $\mathfrak{h}_ 2$ maps leaves of $\mathcal{F}$ to leaves of $\widetilde{\mathcal{F}}$;
\item  $\mathfrak{h}_ 2(\zeta)$ tends to $\widetilde{\mathcal{E}}_2$ as $\zeta$ tends to ${\mathcal{E}}_2$;
\item $\mathfrak{h}_ 2$ maps the set $$R_ 2=\{|y|=b_ 2, 0<|x|\le a_ 2\}$$ onto the set $$\widetilde{R}_ 2=\{|\tilde{y}|=\tilde{b}_ 2, 0<|\tilde{x}|\le \tilde{a}_ 2\}$$ conjugating the one dimensional foliations induced by $\mathcal{F}$ and $\widetilde{\mathcal{F}}$  on $R_ 2$ and $\widetilde{R}_ 2$;
\item $\mathfrak{h}_ 2$ maps each punctured disc $$\{y=u, 0<|x|\le a_ 2\},|u|=b_ 2$$ onto a punctured disc  $$\{\tilde{y}=\tilde{u}, 0<|\tilde{x}|\le \tilde{a}_ 2\}, |\tilde{u}|=\tilde{b}_ 2,$$ so $\mathfrak{h}_ 2$ extends to $\overline{R_ 2}$;
\item close to the divisor and outside $$\{|x|\le\epsilon,|y|\le\epsilon\}\cup\mathfrak{h}^{- 1}\big(\{|\tilde{x}|\le\epsilon,|\tilde{y}|\le\epsilon\}\big)$$ we have $\mathfrak{h}_ 2=\mathfrak{h}$;
\item $\mathfrak{h}_ 2$ induces the same map $\mathfrak{h}^*$.
\end{enumerate}  
In fact, the numbers $a_j,b_j,\tilde{a}_j,\tilde{b}_j$ are arbitrary whenever they are small enough, so we can suppose that  $b_2=b_1$ and $\tilde{b}_2=\tilde{b}_1$. By reducing $W_2$ and $a_2$ if necessary we can assume the following additional properties:
\begin{enumerate}
\item $a_2<a_1$, $\tilde{a}_2<\tilde{a}_1$;
\item $V_1$ and $V_2$ are disjoint;
\item $\mathfrak{h}_1(V_1)$, $\mathfrak{h}_2(V_2)$ and  $\{|x|< \tilde{a}_ 1, |y|< \tilde{b}_ 1\}$ are pairwise disjoint.
\end{enumerate}

For all $s\in[0,1]$, set $\alpha_s=(1-s)a_1+sa_2$, $\tilde{\alpha}_s=(1-s)\tilde{a}_1+s\tilde{a}_2$ and consider the sets  \begin{align*}&T_s=\{|x|=\alpha_s,|y|=b_1\};\\
&\widetilde{T}_s=\{|\tilde{x}|=\tilde{\alpha}_s,|\tilde{y}|=\tilde{b}_1\}.\end{align*} Clearly we have the following:
\begin{enumerate}
\item $\mathfrak{h}_1$ conjugates the one-dimensional foliations on $T_{0}$ and $\widetilde{T}_{0}$;
\item $\mathfrak{h}_2$ conjugates the one-dimensional foliations on $T_{1}$ and $\widetilde{T}_{1}$. 
\end{enumerate}
\begin{lem}There exist  a continuous family of homeomorphisms $$h_s\colon T_s\to \widetilde{T}_s,\; s\in[0,1]$$ with $h_0=\mathfrak{h}_1$, $h_1=\mathfrak{h}_2$ and such that, for each $s\in[0,1]$, the homeomorphism
$h_s$ conjugates the one dimensional foliation induced by $\mathcal{F}$ on $T_s$ with the one dimensional foliation induced by $\widetilde{\mathcal{F}}$ on $\widetilde{T}_s$.
\end{lem}
\begin{proof} Of course, define $h_0=\mathfrak{h}_1$ and $h_1=\mathfrak{h}_2$. Each $T_s$ can be identified with the torus $\partial \mathbb{D}\times\partial\mathbb{D}$ by the map
$$\big(\alpha_se^{2\pi i u},b_1e^{2\pi i v}\big)\mapsto (e^{2\pi i u},e^{2\pi i v});\; u,v\in\mathbb{R}.$$ Then we can think that $h_0$ and $h_1$ are in the class $\mathcal{H}$ of homeomorphisms of $\partial \mathbb{D}\times\partial\mathbb{D}$ preserving the foliation $dv-\lambda du=0$. Clearly it is sufficient to prove that $h_0$ and $h_1$ are included in a  continuous family $\{h_s\}_{s\in[0,1]}$ of homeomorphisms  in $\mathcal{H}$. We know that $h_0$ and $h_1$ lift to some  homeomorphisms $H_1,H_2\colon\mathbb{R}^2\to\mathbb{R}^2$, respectively.  On the other hand, recall that $h_0=\mathfrak{h}_1$ and $h_1=\mathfrak{h}_2$ induce the same map $\mathfrak{h}^*$ at homology level and let us define  $A(u,v)=(u,v)$ or $A(u,v)=(-u,-v)$ according to $\mathfrak{h}^*$ is the identity or the inversion map. Then there exist continuous functions $\kappa_0,\kappa_1\colon\mathbb{R}^2\to\mathbb{R}$ such that (see Lemma \ref{ass})
$$H_j=H_j(0,0)+A(u,v)+\kappa_j(u,v)\cdot (1,\lambda); \; j=0,1.$$ Now, it is not difficult to see that  $$H_s=(1-s)H_0+sH_1; \, s\in[0,1]$$ induce  a continuous family in $\mathcal{H}$. 
\end{proof}
Set $D=\{|x|\le a_1,|y|\le b_1\}$ and $\widetilde{D}=\{|\tilde{x}|\le \tilde{a}_1,|\tilde{y}|\le \tilde{b}_1\}$ and define $\mathfrak{h}_0\colon\partial D\to\partial\widetilde{D}$ as 
\begin{align*} 
\mathfrak{h}_0 &= \mathfrak{h}_1, \textrm{ on } \{|x|=a_1,|y|\le b_1\}; \\
\mathfrak{h}_0 &=\mathfrak{h}_2,\textrm{ on }\{|x|\le a_2,|y|=b_1\};\\
\mathfrak{h}_0 &=h_s, \textrm{ on } T_s,\, s\in [0,1].
\end{align*}
It is easy to see that $\mathfrak{h}_0$ is a homeomorphism conjugating the one-dimensional foliations induced by $\mathcal{F}$ on $\partial D$ and $\widetilde{\mathcal{F}}$ on $\partial\widetilde{D}$. Then, by the conical structure of nodal singularities we can extend $\mathfrak{h}_0$ as a homeomorphism between $D$ and $\widetilde{D}$ mapping leaves of $\mathcal{F}$ to leaves of $\widetilde{ \mathcal{F}}$. Finally, it is easy to see that the map $\bar{\mathfrak{h}}$ defined as 
\begin{align*} 
\bar{\mathfrak{h}} &= \mathfrak{h}_1, \textrm{ on } V_1, \\
\bar{\mathfrak{h}} &=\mathfrak{h}_2,\textrm{ on }V_2, \textrm{ and}\\
\bar{\mathfrak{h}} &=\mathfrak{h}_0, \textrm{ on } D
\end{align*} defines a topological equivalence between $\mathcal{F}$ and $\widetilde{\mathcal{F}}$ extending to the nodal corner singularity $p$. Moreover,
close to the divisor and outside $$\{|x|\le\epsilon,|y|\le\epsilon\}\cup\mathfrak{h}^{- 1}(\{|\tilde{x}|\le\epsilon,|\tilde{y}|\le\epsilon\})$$ we have $\bar{\mathfrak{h}}=\mathfrak{h}$. This last property permit us to  repeat finitely many times the construction  above    to obtain a topological equivalence satisfying the requirements of Theorem \ref{nodalext}.

\end{document}